\documentclass[11pt]{article}
\usepackage{amsmath}
\usepackage{amsthm}
\usepackage{enumitem}
\usepackage{hyperref}

\newcommand{\C}{\mathbf{C}}
\newcommand{\R}{\mathbf{R}}
\newcommand{\Q}{\mathbf{Q}}

\newcommand{\N}{\mathbf{N}}
\newcommand{\tto}{\Rightarrow}
\newcommand{\eps}{\epsilon}

\DeclareMathOperator\p{\mathbf{P}}
\DeclareMathOperator\E{\mathbf{E}}
\DeclareMathOperator\I{\mathbf{1}}
\DeclareMathOperator\Var{\mathbf{Var}}

\DeclareMathOperator\rank{rank}
\DeclareMathOperator\tr{tr}

\newcommand{\semicolon}{\mathrel{;}}

\theoremstyle{plain}
\newtheorem{theorem}{Theorem}[section]
\newtheorem{proposition}[theorem]{Proposition}
\newtheorem{lemma}[theorem]{Lemma}
\newtheorem{corollary}[theorem]{Corollary}

\theoremstyle{definition}
\newtheorem{definition}[theorem]{Definition}

\theoremstyle{remark}
\newtheorem{remark}[theorem]{Remark}

\numberwithin{equation}{section}

\setenumerate[1]{label={\rm(\roman*)},ref={\rm(\roman*)}}
\setenumerate[2]{label={\rm(\alph*)},ref={\rm\theenumi (\alph*)}}
\setenumerate[3]{label={\rm\arabic*}.,ref={\rm\theenumii\arabic*}}

\newcommand{\esdn}{\mu_{W_n}}
\newcommand{\scdist}{\mu_{\mathrm{sc}}}
\newcommand{\wn}{w_{ij}^{(n)}}
\newcommand{\Varwn}{\Var\bigl[\wn\bigr]}
\newcommand{\lindn}{\E\bigl[\,|\wn|^2\semicolon|\wn|>\eps\,\bigr]}
\newcommand{\on}[1]{\lim_{n\to\infty}\frac{1}{n}#1=0}
\newcommand{\bi}{\mathbf{i}}
\newcommand{\bj}{\mathbf{j}}
\newcommand{\bc}{\mathbf{c}}
\newcommand{\norm}[1]{\lVert#1\rVert}
\newcommand{\fnorm}[1]{\lVert#1\rVert_{F}}
\newcommand{\tvnorm}[1]{\lVert#1\rVert_{TV}}
\newcommand{\infnorm}[1]{\lVert#1\rVert_\infty}
\newcommand{\inner}[2]{\langle#1,#2\rangle}
\newcommand{\detail}[1]{(#1)}

\begin{document}
\title{An Exposition on Wigner's Semicircular Law}
\author{Wooyoung Chin}
\maketitle
\abstract{
	We revisit the moment method to obtain a slightly strengthened version
	of the usual semicircular law.
	Our version assumes only that the upper triangular entries of Hermitian
	random matrices are independent, have mean zero and variances close to
	$1/n$ in a certain sense, and satisfy a Lindeberg-type condition.
	As an application, we derive another semicircular law for the case when
	the sum of a row converges in distribution to the standard normal
	distribution, including the case where all matrix entries may have
	infinite variance.
	The appendix, making up the majority of the paper, provides for those
	new to the subject, a rigorous exposition of most details involved,
	including also a proof of a semicircular law that uses the
	Stieltjes transform method.
}

\tableofcontents

\section{Introduction}

If $A$ is an $n \times n$ Hermitian matrix, then $A$ is diagonalizable
and all eigenvalues of $A$ are real.
We denote the eigenvalues of $A$, counted with multiplicities, as
\[ \lambda_1(A) \ge \cdots \ge \lambda_n(A). \]
\detail{For details, see Subsection \ref{subsec:spectra_basic}.}
We define the \emph{spectral distribution} of $A$ as the Borel probability
measure
\[ \mu_A := \frac{1}{n} \sum_{i=1}^n \delta_{\lambda_i(A)} \]
on $\R$.
If $X$ is a random $n \times n$ Hermitian matrix, then
$\mu_X$ is a random Borel probability measure on $\R$.
\detail{For details, see Subsection \ref{subsec:rand_hermitian}.}

\begin{definition}[The semicircle distribution]
	The Borel probability measure $\scdist$ on $\R$ given by
	\begin{equation*}
		\scdist(dx) = \frac{1}{2\pi} \sqrt{(4-x^2)_+}\,dx
	\end{equation*}
	is called the \emph{semicircle distribution}.
	Here, $x_+ := x \vee 0 = \max\{x,0\}$.
\end{definition}

Since the seminal work \cite{Wig55} by Wigner, there have been many theorems
that assume $W_1,W_2,\ldots$ to be random $1 \times 1$, $2 \times 2, \ldots$
Hermitian matrices satisfying certain conditions, and show that $\mu_{W_n}$
converges in some sense to $\scdist$.
Let us call such theorems \emph{semicircular laws}.

In the main part of this paper, we study semicircular laws assuming joint
independence of the upper triangular entries (we include the diagonal in
both the upper and the lower triangles).
We first prove a semicircular law (Theorem \ref{thm:sclaw})
with rather weak assumptions.
In particular, we don't require the entries to be identically distributed,
and we allow the entries to deviate from unit variance.
It is notable that other than the mostly standard reduction steps, our proof is
just a simple application of the moment method.

After proving the main theorem, we apply the theorem to obtain
another semicircular law (Theorem \ref{thm:gauss_sclaw})
which more or less assumes that the sum of a row
converges in distribution to the standard normal distribution.
This theorem allows the entries to have infinite variances.

The appendices make up about two thirds of this paper.
There we provide a self-contained and rigorous account of the details
(including the measure-theoretic ones) involved in the main part of the paper.
In the main part of the paper, we refer to the appendix whenever we need
a fact given there.
After that, for completeness we provide a proof of a semicircular law
(little weaker than the one proved in the main part, but still stronger than
the laws appear in many textbooks)
which uses the Stieltjes transform method.

We assumed no prior knowledge more advanced than one-semester courses
in probability theory and combinatorics.
A total newcomer to the field might want to read the Appendices
\ref{sec:prob_meas_r}--\ref{sec:concen_meas} first,
then read the main part, and then go to Appendix
\ref{sec:unit_reduction}--\ref{sec:stieltjes}.

Now we state our main theorem.

\begin{theorem}[A general semicircular law]
\label{thm:sclaw}
	For each $n \in \N$, let $W_n = (\wn)_{i,j=1}^n$ be a random $n \times n$
	Hermitian matrix \detail{see Definition \ref{def:rand_hermitian}}
	whose upper triangular entries are jointly independent, have mean zero,
	and have finite variances.
	We assume that $W_1, W_2, \ldots$ are defined on the same
	probability space.
	If
	\begin{equation}
		\label{eq:sclaw_var}
		\on{
			\sum_{i=1}^n \biggl|
				\sum_{j=1}^n \Bigl( \Varwn - \frac{1}{n} \Bigr)
			\biggr|
		},
	\end{equation}
	\begin{equation}
		\label{eq:sclaw_rowbdd}
		\on{
			\sum_{i=1}^n \biggl( \sum_{j=1}^n \Varwn - C \biggr)_+
		} \qquad \text{for some finite $C \ge 0$,}
	\end{equation}
	and
	\begin{equation}
		\label{eq:sclaw_lind}
		\on{ \sum_{i,j=1}^n \lindn }
		\qquad \text{for every $\eps>0$,}
	\end{equation}
	then $\esdn \tto \scdist$ as $n \to \infty$ a.s.
\end{theorem}

\begin{remark}
	One sufficient condition for \eqref{eq:sclaw_var} and \eqref{eq:sclaw_rowbdd}
	to hold is
	\[ \on{ \sum_{i,j=1}^n \Bigl|\Varwn-\frac{1}{n}\Bigr| }. \]
	Note that we can take $C = 1$ to show \eqref{eq:sclaw_rowbdd}.
	An even more special case is when we have $\Varwn = 1/n$
	for all $n = 1,2,\ldots$ and $i,j = 1,\ldots,n$.
	This case is Theorem 2.9 in \cite{BS10}.
	If there is some finite $C \ge 0$ such that
	\begin{equation}
		\label{eq:abs_rowbdd}
		\sum_{j=1}^n \Varwn \le C \qquad
		\text{for all $n = 1,2,\ldots$ and $i = 1,\ldots,n$,}
	\end{equation}
	then \eqref{eq:sclaw_rowbdd} holds for the same $C$.
	This case is more or less equivalent to Corollary 1 in \cite{GNT15},
	which is proved first for matrices with Gaussian entries,
	and then generalized to arbitrary matrices by proving an analogue
	of the Lindeberg universality principle for random matrices.
	In this paper, we will prove Theorem \ref{thm:sclaw} directly
	by the moment method without appealing to the universality principle.
\end{remark}

\begin{remark}
	Theorem \ref{thm:sclaw} assumes no dependence between $W_1$, $W_2$, \ldots,
	yet it asserts an a.s.\ convergence.
	This is in contrast to some versions of the semicircular law
	where only convergence in probability is asserted
	(e.g. \cite[Theorem 2.1.1]{AGZ10}), or $\sqrt{n}W_n$ is assumed to be
	the top left $n \times n$ minor of a fixed infinite random Hermitian matrix
	(e.g. \cite[Theorem 2.4.2]{Tao12}).
	If $\mu_1,\mu_2,\ldots$ are Borel probability distributions on a
	separable metric space $S$, and $c \in S$, then the following
	two statements are equivalent:
	\begin{enumerate}
		\item
		$X_n \to c$ a.s.\ whenever $X_1,X_2,\ldots$ are random elements of $S$
		defined on a common probability space such that each $X_n$
		has distribution $\mu_n$;
		\item
		$\sum_{n=1}^\infty \mu_n\bigl(\{\, x \in S \mid d(x,c) > \eps \,\}\bigr)
		< \infty$ for all $\eps > 0$.
	\end{enumerate}
	This can be shown using the Borel-Cantelli lemmas
	\cite[Theorem 4.3 and 4.4]{Bil12}.
	This type of strong convergence is possible in Theorem \ref{thm:sclaw}
	because of a strong concentration of measure result we will use.
\end{remark}

The rest of the paper is organized as follows.
In Section \ref{sec:prelim_reductions}, we will first reduce Theorem
\ref{thm:sclaw} to a form with stronger assumptions.
Then we will see that the reduced semicircular law follows from
some moment computations.
In Section \ref{sec:trees}, we will develop a tool needed for the moment
computation, and in Section \ref{sec:comp_moments}, we will perform
the actual moment computation.
In Section \ref{sec:gaussian}, we will derive the aforementioned
semicircular law which assumes Gaussian convergence of the sum of a row.

\section{Preliminary reductions}
\label{sec:prelim_reductions}
Assume that $W_n$ satisfies the conditions of Theorem \ref{thm:sclaw}.

\subsection{Convergence in expectation is enough}
If we have $\E \esdn \tto \scdist$ \detail{for the meaning of $\E\mu_{W_n}$,
see Theorem \ref{thm:expect_prob}}, then
\[ \lim_{n\to\infty} \E\biggl[\int_\R f\,d\esdn\biggr] = \int_\R f\,d\scdist \]
for all continuous and bounded $f\colon\R\to\R$.
By the concentration inequality Theorem \ref{thm:concen_spec} for spectral
measures and the Borel-Cantelli lemma, we have
\[ \lim_{n\to\infty} \int_\R f_{p,q}\,d\esdn = \int_\R f_{p,q}\,d\scdist
\qquad \text{a.s.}\]
for all $p,q \in \Q$ with $p < q$, where $f_{p,q}\colon\R\to\R$ is
$1$ on $(-\infty,p\,]$, $0$ on $[\,q,\infty)$, and linear on $[\,p,q\,]$.
This implies that $\esdn \tto \scdist$ a.s.\ by Theorem \ref{thm:ch_weak_conv}.
Therefore, it is enough to show $\E\esdn \tto \scdist$.

\subsection{Truncation}
Since \eqref{eq:sclaw_lind} holds, we have positive integers
$n_1 < n_2 < \ldots$ such that
\[ \frac{1}{n} \sum_{i,j=1}^n \E\bigl[\,|\wn|^2\semicolon|\wn|>1/k\,\bigr]
\le 1/k\]
for all $n \ge n_k$ for each $k \in \N$.
If we let $\eta_n = 1$ for all $n \in \{1,\ldots,n_1-1\}$,
and $\eta_n = 1/k$ for all $n \in \{n_k,\ldots,n_{k+1}-1\}$
for each $k \in \N$, then $\eta_n \to 0$ and
\[ \on{ \sum_{i,j=1}^n \E\bigl[\,|\wn|^2\semicolon|\wn|>\eta_n\,\bigr] }. \]

Let $W_n' := \bigl(\wn\I(|\wn| \le \eta_n)\bigr)_{i,j=1}^n$.
Since
\[ \frac{1}{n}\E\bigl[\tr(W_n-W_n')^2\bigr]
= \frac{1}{n}\sum_{i,j=1}^n \E\bigl[\,|\wn|^2 \semicolon |\wn|>\eta_n\,\bigr]
\to 0 \qquad \text{as $n \to \infty$}, \]
it is enough to show $\E\mu_{W'_n} \tto \scdist$
by Corollary \ref{cor:pertrub_frob_norm_exp} and Theorem \ref{thm:ch_weak_conv}.

\subsection{Centralization}
For each $n = 1,2,\ldots$ and $i,j = 1,\ldots,n$, let
\[ v_{i,j}^{(n)} := \wn\I(|\wn| \le \eta_n)
- \E\bigl[\,\wn \semicolon |\wn| \le \eta_n\,\bigr], \]
and let $W_n' - \E W_n' := (v_{ij}^{(n)})_{i,j=1}^n$.
Since
\begin{multline*}
	\frac{1}{n}\sum_{i,j=1}^n \bigl|\E\bigl[\,
	\wn \semicolon |\wn|\le\eta_n \,\bigr]\bigr|^2
	= \frac{1}{n}\sum_{i,j=1}^n \bigl|\E\bigl[\,
	\wn \semicolon |\wn|>\eta_n \,\bigr]\bigr|^2 \\
	\le \frac{1}{n}\sum_{i,j=1}^n \E\bigl[\,
	|\wn|^2 \semicolon |\wn|>\eta_n \,\bigr]
	\to 0 \qquad \text{as $n \to \infty$,}
\end{multline*}
it is enough to show $\E\mu_{W'_n-\E W'_n} \tto \scdist$ by
Corollary \ref{cor:pertrub_frob_norm_exp} and Theorem \ref{thm:ch_weak_conv}.

We claim that $W'_n-\E W'_n$ satisfies all conditions
$W_n$ is supposed to satisfy in Theorem \ref{thm:sclaw}.
The fact that \eqref{eq:sclaw_var} and \eqref{eq:sclaw_rowbdd}
still hold even if we replace $W_n$ by $W'_n-\E W'_n$ follows
from the following:
\begin{equation*}
	\begin{split}
		\sum_{i,j=1}^n &\bigl|\Varwn - \Var\bigl[v_{ij}^{(n)}\bigr]\bigr| \\
		&= \sum_{i,j=1}^n \Bigl( \E\bigl[\, |\wn|^2 \semicolon
		|\wn| > \eta_n \,\bigr]
		+ \bigl|\E\bigl[\, \wn \semicolon
		|\wn| \le \eta_n \,\bigr]\bigr|^2 \Bigr) \\
		&\le 2 \sum_{i,j=1}^n \E\bigl[\, |\wn|^2 \semicolon
		|\wn| > \eta_n \,\bigr].
	\end{split}
\end{equation*}
The condition \eqref{eq:sclaw_lind} for $W'_n-\E W'_n$ easily follows
from the bound $|\wn| \le \eta_n$.
Since $|v_{i,j}^{(n)}| \le 2\eta_n$ for all
$n = 1,2,\ldots$ and $i,j = 1,\ldots,n$, by doubling $\eta_1$, $\eta_2$, \ldots
we have $\eta_n \to 0$ and $|v_{i,j}^{(n)}| \le \eta_n$.
Thus, from now on, we can assume that $|\wn| \le \eta_n$ for some
$\eta_1, \eta_2, \ldots > 0$ satisfying $\eta_n \to 0$.

\subsection{Rescaling}
Fix $n \in \N$.
We will choose a number $0 \le c_{ij}^{(n)} \le 1$ for each $i,j = 1,\ldots,n$
so that $c_{ij}^{(n)} = c_{ji}^{(n)}$ always hold.
Start by letting $c_{ij}^{(n)} = 1$ for all $i,j = 1,\ldots,n$.
We start with the first row and the first column.
If $\sum_{j=1}^{n} \Var\bigl[w_{1j}^{(n)}\bigr] \le C$, then do nothing.
Otherwise, lower $c_{11}^{(n)},\ldots,c_{1n}^{(n)}$ (not below $0$) so that
\[ \sum_{j=1}^n \Bigl[\bigl(c_{1j}^{(n)}\bigr)^2
\Var\bigl[w_{1j}^{(n)}\bigr]\Bigr] = C, \]
and let $c_{j1}^{(n)} := c_{1j}^{(n)}$ for all $j = 1,\ldots,n$.
We note that at this point we have
\[ \sum_{i,j=1}^{n} \Bigl[\bigl(1 - \bigl(c_{ij}^{(n)}\bigr)^2\bigr)
\Var\bigl[w_{ij}^{(n)}\bigr]\Bigr]
\le 2 \biggl(\sum_{j=1}^n \Var\bigl[ w_{1j}^{(n)} \bigr] - C\biggr)_+. \]

Assume that $k \in \{2,\ldots,n\}$, and that we've examined up to
$(k-1)$-th row.
If
\[\sum_{j=1}^{k-1} \Bigl[\bigl(c_{kj}\bigr)^2\Var\bigl[w_{kj}^{(n)}\bigr]\Bigr]
+ \sum_{j=k}^n \Var\bigl[w_{kj}^{(n)}\bigr] \le C, \]
then do nothing.
Otherwise, lower $c_{k1}^{(n)},\ldots,c_{kn}^{(n)}$ (not below $0$) so that
\[ \sum_{j=1}^n \Bigl[\bigl(c_{kj}^{(n)}\bigr)^2
\Var\bigl[w_{kj}^{(n)}\bigr]\Bigr] = C, \]
and let $c_{jk}^{(n)} = c_{kj}^{(n)}$ for all $j = 1,\ldots,n$.
At this point we have
\[ \sum_{i,j=1}^{n} \Bigl[\bigl(1 - \bigl(c_{ij}^{(n)}\bigr)^2\bigr)
\Var\bigl[w_{ij}^{(n)}\bigr]\Bigr]
\le 2 \sum_{i=1}^k \biggl(\sum_{j=1}^n
\Var\bigl[ w_{ij}^{(n)} \bigr] - C\biggr)_+. \]
This can be shown by an induction on $k$.
After completing the whole process, we are left with numbers
$0 \le c_{ij}^{(n)} \le 1$ such that
\begin{equation}
	\label{eq:rescale_rowbdd}
	\sum_{j=1}^n \Bigl[\bigl(c_{ij}^{(n)}\bigr)^2
	\Var\bigl[w_{ij}^{(n)}\bigr]\Bigr] \le C
\end{equation}
for all $i = 1,\ldots,n$, and
\begin{equation}
	\label{eq:rescale_change}
	\sum_{i,j=1}^{n} \Bigl[\bigl(1 - \bigl(c_{ij}^{(n)}\bigr)^2\bigr)
	\Var\bigl[w_{ij}^{(n)}\bigr]\Bigr]
	\le 2 \sum_{i=1}^n \biggl(\sum_{j=1}^n
	\Var\bigl[ w_{ij}^{(n)} \bigr] - C\biggr)_+.
\end{equation}

Let $\widehat{W}_n = \bigl( c_{ij}^{(n)} w_{ij}^{(n)} \bigr)_{i,j=1}^n$.
Since $(1-c)^2 \le 1-c^2$ holds for any $0 \le c \le 1$, we have
\begin{multline*}
	\frac{1}{n} \E\bigl[ \tr(W_n - \widetilde{W}_n)^2 \bigr]
	= \frac{1}{n}\sum_{i,j=1}^n \Bigl[
	\bigl(1 - c_{ij}^{(n)}\bigr)^2 \Varwn \Bigr] \\
	\le \frac{2}{n} \sum_{i=1}^n \biggl(\sum_{j=1}^n
	\Var\bigl[ w_{ij}^{(n)} \bigr] - C\biggr)_+
	\to 0 \qquad \text{as $n \to \infty$}
\end{multline*}
by \eqref{eq:rescale_change}.
Thus, by Corollary \ref{cor:pertrub_frob_norm_exp} and Theorem
\ref{thm:ch_weak_conv}, it is enough to show
$\E\mu_{\widehat{W}_n} \tto \scdist$.

The altered matrix $\widehat{W}_n$ has an advantage over $W_n$ that
\eqref{eq:rescale_rowbdd} holds.
Also, the modulus of each entry of $\widehat{W}_n$ is still bounded by $\eta_n$.
We claim that $\widehat{W}_n$ also satisfies all conditions $W_n$ is assumed to
satisfy in Theorem \ref{thm:sclaw}.
First, each entry of $\widehat{W}_n$ obviously has mean zero.
Also, since each entry of $\widehat{W}_n$ has modulus less than or equal to
the corresponding entry of $W_n$, the condition
\eqref{eq:sclaw_lind} is satisfied by $\widehat{W}_n$.
The condition \eqref{eq:sclaw_rowbdd} for $\widehat{W}_n$ obviously holds
as we have an even stronger property \eqref{eq:rescale_rowbdd}.
Finally, \eqref{eq:sclaw_var} for $\widehat{W}_n$ follows from
\eqref{eq:rescale_change} and the fact that \eqref{eq:sclaw_rowbdd} is
satisfied by $W_n$.
This proves our claim, and so from now on, we can also assume that
\eqref{eq:abs_rowbdd} is true.

\subsection{Reduction to moment convergence}
On top of the assumptions of Theorem \ref{thm:sclaw}, we now also have
the following.
\begin{enumerate}
	\item
	There are $\eta_1,\eta_2,\ldots > 0$ with $\lim_{n\to\infty} \eta_n = 0$
	such that $|\wn| \le \eta_n$ for all $n = 1,2,\ldots$ and
	$i,j = 1,\ldots,n$.
	\item
	There is some finite $C \ge 0$ such that \eqref{eq:abs_rowbdd} holds.
\end{enumerate}
Since $|\wn| \le \eta_n$, every eigenvalue of $W_n$ has
absolute value at most $n\eta_n$.
So, $\E\esdn$ is supported on $[-n\eta_n,n\eta_n]$,
and in particular $\E\esdn$ has moments of all orders.
As
\begin{equation*}
	\biggl|\sum_{k=1}^\infty \frac{1}{k!} \int_\R x^k \,\scdist(dx) r^k \biggr|
	\le \sum_{k=1}^\infty \frac{|2r|^k}{k!} < \infty
\end{equation*}
for any $r \in \R$ by the ratio test, $\scdist$ is determined by its moments
by \cite[Theorem 30.1]{Bil12}.
Thus, by the moment convergence theorem \cite[Theorem 30.2]{Bil12},
it is enough to show
\begin{equation*}
	\lim_{n \to \infty} \int_\R x^k \,\E\esdn(dx) = \int_\R x^k \,\scdist(dx)
	\qquad \text{for all $k = 1,2,\ldots$.}
\end{equation*}

For each $k = 1,2,\ldots$, since there are continuous bounded
$g_{k,n}: \R \to \R$ with $g_{k,n}(x) = x^k$ for all $x \in [-n\eta_n,n\eta_n]$,
we have
\begin{equation*}
	\int_\R x^k \,\E\esdn(dx) = \int_\R g_{k,n} \,d\E\esdn
	= \E \int_\R g_{k,n} \,d\esdn = \frac{1}{n} \E \tr W_n^k.
\end{equation*}
On the other hand, we can directly compute the moments of $\scdist$ as follows.

\begin{lemma}
	For any $m=1,2,\ldots$, we have
	\begin{equation*}
		\int_\R x^{2m} \,\scdist(dx) = \frac{1}{m+1} \binom{2m}{m}.
	\end{equation*}
\end{lemma}

\begin{proof}
	A trigonometric substitution $x = 2\cos\theta$ gives
	\begin{equation*}
		\begin{split}
			\int_\R x^{2m} \,\scdist(dx)
			&= \frac{1}{2\pi} \int_{-2}^2 x^{2m} \sqrt{4-x^2} \,dx
			= \frac{2}{\pi} \int_{-\pi}^0 2^{2m}
			\cos^{2m}\theta \sin^2\theta \,d\theta\\
			&= \frac{2^{2m+1}}{\pi}
			\left[ \int_{-\pi}^0 \cos^{2m}\theta \,d\theta
			- \int_{-\pi}^0 \cos^{2m+2}\theta \,d\theta \right].
		\end{split}
	\end{equation*}
	As
	\begin{equation*}
		\int_{-\pi}^0 \cos^{2l} \theta \,d\theta
		= \frac{1}{2^{2l+1}} \int_{-\pi}^\pi
		(e^{i\theta}+e^{-i\theta})^{2l} \,d\theta
		= \frac{\pi}{2^{2l}} \binom{2l}{l}
	\end{equation*}
	for any $l=1,2,\ldots$, we have
	\begin{equation*}
		\int_\R x^{2m} \,\scdist(dx)
		= 2\binom{2m}{m} - \frac{1}{2} \binom{2m+2}{m+1}
		= \frac{1}{m+1} \binom{2m}{m}.
	\end{equation*}
\end{proof}

Note that $\int_\R x^k \,\scdist(dx) = 0$ whenever $k \in \N$ is odd.
Thus, it is enough to show that
\begin{equation}
	\label{eq:odd_moment_conv}
	\lim_{n\to\infty} \frac{1}{n} \E \tr W_n^k = 0
	\qquad \text{for all odd $k \in \N$,}
\end{equation}
and that
\begin{equation}
	\label{eq:even_moment_conv}
	\lim_{n\to\infty} \frac{1}{n} \E \tr W_n^k = \frac{1}{k/2+1} \binom{k}{k/2}
	\qquad \text{for all even $k \in \N$.}
\end{equation}
These will be proved in Section \ref{sec:comp_moments}
by using the content of Section \ref{sec:trees}.

\section{Trees and products of variances}
\label{sec:trees}
Our graphs will be undirected.
We allow graphs to have loops, but don't allow them to have multiple edges.
Let $G$ be a finite graph.
For any $n \in \N$, denote by $I(G,n)$ the collection of all injections
from $V(G)$ into $\{1,\ldots,n\}$.
Given any $F \in I(G,n)$ and $e \in E(G)$ with ends $u,v$,
we let
\[\rho_{e,F}^{(n)} := \Var \bigl[ w_{F(u)F(v)}^{(n)} \bigr].\]
It is well-defined since each $W_n$ is Hermitian.
Then we let
\begin{equation*}
	P(G,F) := \prod_{e \in E(G)} \rho_{e,F}^{(n)}.
\end{equation*}
Here $P$ stands for ``product.'' Also, the notation $\rho_{e,F}^{(n)}$
will no longer appear.

\begin{lemma}
	\label{lem:tree_bdd}
	If $T$ is a finite tree with $m$ edges, $u \in V(T)$, $n \in \N$,
	and $i \in \{1,\ldots,n\}$, then
	\begin{equation}
		\label{eq:tree_bdd}
		\sum_{\substack{F \in I(T,n)\\F(u)=i}} P(T,F) \le C^m.
	\end{equation}
\end{lemma}

\begin{proof}
	If $m = 0$, then \eqref{eq:tree_bdd} obviously holds.
	(We define the product of zero terms as $1$.)
	To proceed by induction, assume that \eqref{eq:tree_bdd} holds for $m$,
	and let $T$ be a tree with $m+1$ edges.
	Choose any leaf $w$ of $T$ different from $u$,
	and let $x$ be the only vertex of $T$ adjacent to $w$.
	Since
	\begin{equation*}
		\begin{split}
			\sum_{\substack{F \in I(T,n)\\F(u)=i, F(x)=j}} P(T,F)
			&\le \sum_{\substack{H \in I(T\setminus w,n)\\H(u)=i, H(x)=j}}
			\Bigl(P(T\setminus w, H) \sum_{l=1}^n
			\Var \bigl[w_{jl}^{(n)}\bigr]\Bigr)\\
			&\le C\sum_{\substack{H \in I(T\setminus w,n)\\H(u)=i, H(x)=j}}
			P(T\setminus w, H)
		\end{split}
	\end{equation*}
	for all $j \in \{1,\ldots,n\}$, we have
	\begin{equation*}
		\begin{split}
			\sum_{\substack{F \in I(T,n)\\F(u)=i}} P(T,F)
		&= \sum_{j=1}^n
		\sum_{\substack{F \in I(T,n)\\F(u)=i, F(x)=j}} P(T,F) \\
		&\le C \sum_{j=1}^n
		\sum_{\substack{H \in I(T\setminus w,n)\\H(u)=i, H(x)=j}}
		P(T\setminus w, H) \\
		&= C \sum_{\substack{H \in I(T\setminus w,n)\\H(u)=i}}
		P(T\setminus w, H) \le C^{m+1}
		\end{split}
	\end{equation*}
	by the induction hypothesis.
\end{proof}

\begin{lemma}
	\label{lem:tree_cont}
	For any finite tree $T$,
	\begin{equation}
		\label{eq:tree_cont}
		\lim_{n\to\infty} \frac{1}{n}
		\sum_{F \in I(T,n)} P(T,F) = 1.
	\end{equation}
\end{lemma}

\begin{proof}
	Let $m := E(T)$.
	If $m = 0$, then \eqref{eq:tree_cont} obviously holds.
	To proceed by induction, assume that the result holds for trees with
	$m$ edges, and let $T$ be a tree with $m+1$ edges.
	Let $u \in V(T)$ be a leaf of $T$, and $w$ be the only vertex of $T$
	adjacent to $u$ in $T$.
	Note that
	\begin{multline}
		\label{eq:tree_cont_free}
		\Biggl| \frac{1}{n} \sum_{F \in I(T,n)} P(T,F)
		- \frac{1}{n} \sum_{H \in I(T\setminus u, n)} \biggl( P(T\setminus u, H)
		\sum_{i=1}^n \Var\bigl[w_{H(w)i}^{(n)}\bigr] \biggr) \Biggr| \\
		\le \frac{1}{n} \sum_{H \in I(T\setminus u, n)}
		\biggl( P(T\setminus u, H) \sum_{v \in V(T\setminus u)}
		\Var\bigl[w_{H(w)H(v)}^{(n)}\bigr] \biggr) \\
		\le \frac{(m+1)\eta_n^2}{n} \sum_{H \in I(T\setminus u,n)}
		P(T\setminus u, H) \to 0 \qquad \text{as $n \to \infty$,}
	\end{multline}
	by the induction hypothesis.
	By Lemma \ref{lem:tree_bdd}, we have
	\begin{multline}
		\label{eq:tree_cont_bdd}
		\biggl| \frac{1}{n} \sum_{H \in I(T\setminus u, n)}
		\biggl( P(T\setminus u, H)
		\sum_{i=1}^n \biggl( \Var\bigl[w_{H(w)i}^{(n)}\bigr] - \frac{1}{n}
		\biggr) \biggr) \biggr|\\
		= \frac{1}{n} \biggl| \sum_{j=1}^n \biggl[ \sum_{i=1}^n
		\biggl(\Var\bigl[w_{ji}^{(n)} \bigr]
		- \frac{1}{n} \biggr) \cdot
		\sum_{\substack{H \in I(T\setminus u,n)\\H(w)=j}}
		P(T\setminus u, H) \biggr] \biggr|\\
		\le \frac{C^m}{n} \sum_{j=1}^n \biggl| \sum_{i=1}^n \biggl(
		\Var\bigl[w_{ij}^{(n)}\bigr] - \frac{1}{n} \biggr) \biggr|
		\to 0 \qquad \text{as $n \to \infty$.}
	\end{multline}
	Combining \eqref{eq:tree_cont_free}, \eqref{eq:tree_cont_bdd},
	and the fact that
	\[ \lim_{n\to\infty} \frac{1}{n} \sum_{H \in I(T\setminus u,n)}
	P(T\setminus u, H) = 1, \]
	we can conclude that \eqref{eq:tree_cont} holds.
\end{proof}

\section{Computation of moments}
\label{sec:comp_moments}
Fix a $k \in \N$.
Let us call any $n$-tuple $(i_0,\ldots,i_k)$ with $i_0 = i_k$
a \emph{closed walk of length $k$}.
If $\bi = (i_0,\ldots,i_k)$ is a closed walk, we let $G(\bi)$ be the graph
(possibly having loops but having no multiple edges) with the vertex set
$V(\bi) := \{i_0,\ldots,i_k\}$ and the edge set
\begin{equation*}
	E(\bi) := \bigl\{\, \{i_{t-1},i_t\} \bigm| t = 1,\ldots,k \,\bigr\}.
\end{equation*}
Two closed walks $\bi = (i_0,\ldots,i_k)$ and $\bj = (j_0,\ldots,j_k)$
are said to be \emph{isomorphic} if for any $s,t = 0,\ldots,k$ we have
$i_s = i_t$ if and only if $j_s = j_t$.
If $t \in \N$, then a \emph{canonical closed walk of length $k$ on $t$ vertices}
is a closed walk $\bc = (c_0,\ldots,c_k)$ with $V(\bc)=\{1,\ldots,t\}$ such that
\begin{enumerate}
	\item
	$c_0 = c_k = 1$ and
	\item
	$c_t \le \max\{c_0,\ldots,c_{t-1}\}+1$ for each $t = 1,\ldots,k$.
\end{enumerate}
Let $\Gamma(k,t)$ denote the set of such walks.
It is straightforward to show that any closed walk is isomorphic to exactly one
canonical closed walk.
For any $\bc \in \Gamma(k,t)$, let $L(n,\bc)$ denote the set of all closed walks
$(i_0,\ldots,i_k)$ with $i_0,\ldots,i_k \in \{1,\ldots,n\}$
which are isomorphic to $\bc$.

Note that
\begin{equation}
	\label{eq:trace_sum}
	\begin{split}
		\frac{1}{n} \E \tr W_n^k &= \frac{1}{n} \sum_{i_0,\ldots,i_k = 1}^n
		\E\biggl[ \prod_{s=1}^k w_{i_{s-1}i_s}^{(n)} \biggr] \\
		&= \sum_{t=1}^{k+1} \sum_{\bc \in \Gamma(k,t)}
		\sum_{(i_0,\ldots,i_k) \in L(n,\bc)}
		\E\biggl[ \prod_{s=1}^k w_{i_{s-1}i_s}^{(n)} \biggr].
	\end{split}
\end{equation}
Here the upper bound of $t$ is (rather arbitrarily) set to $k+1$
since $\Gamma(k,t)$ is empty for any $t > k+1$.
We will compute
\[ \sum_{(i_0,\ldots,i_k) \in L(n,\bc)}
\E\biggl[ \prod_{s=1}^k w_{i_{s-1}i_s}^{(n)} \biggr] \]
for each $t \in \N$ and $\bc \in \Gamma(k,t)$.

\begin{lemma}
	\label{lem:no_walk_once}
	Let $t \in \N$ and $\bc = (c_0,\ldots,c_k) \in \Gamma(k,t)$.
	If $\bc$ walks on some edge $\{i,j\}$ exactly once,
	i.e. $\{c_{s-1},c_s\} = \{i,j\}$ for exactly one $s \in \{1,\ldots,k\}$,
	then
	\[ \E \biggl[ \prod_{s=1}^k w_{i_{s-1}i_s}^{(n)} \biggr] = 0 \]
	for any $n \in \N$ and $(i_0,\ldots,i_k) \in L(n,\bc)$.
\end{lemma}

\begin{proof}
	Since the upper triangular entries of $W_n$ are jointly independent,
	$\prod_{s=1}^k w_{i_{s-1}i_s}^{(n)}$ can be broken into
	$\wn$ or $w_{ji}^{(n)}$, and a random variable independent from $w_{ij}$.
	Since $\E \wn = 0$, the desired conclusion follows.
\end{proof}

\begin{lemma}
	\label{lem:no_mult_cycle}
	Let $t \in \N$ and $\bc = (c_0,\ldots,c_k) \in \Gamma(k,t)$.
	Assume that $\bc$ doesn't walk on any edge exactly once,
	i.e. for each $s = 1,\ldots,k$ there is a $r \in \{1,\ldots,k\}$
	with $r \ne s$ such that $\{c_{s-1},c_s\} = \{c_{r-1},c_r\}$.
	Then we have $t \le k/2+1$, and the following hold.
	\begin{enumerate}
		\item
		If $t < k/2 + 1$, then
		\begin{equation}
			\label{eq:iso_sum_zero}
			\on{ \sum_{(i_0,\ldots,i_k) \in L(n,\bc)} \E \biggl[
			\prod_{s=1}^k w_{i_{s-1}i_s}^{(n)} \biggr] }.
		\end{equation}
		\item
		If $t = k/2 + 1$, then
		\begin{equation}
			\label{eq:iso_sum_one}
			\lim_{n\to\infty} \frac{1}{n}
			\sum_{(i_0,\ldots,i_k) \in L(n,\bc)} \E \biggl[
			\prod_{s=1}^k w_{i_{s-1}i_s}^{(n)} \biggr] = 1.
		\end{equation}
	\end{enumerate}
\end{lemma}

\begin{proof}
	As each edge of $G(\bc)$ is walked on at least twice by $\bc$,
	the graph $G(\bc)$ has at most $k/2$ edges.
	Since $G(\bc)$ is a connected graph with $t$ vertices, we have
	$t \le k/2+1$, and $G(\bc)$ has a spanning tree $S$ with $t-1$ edges.
	If $\bi = (i_0,\ldots,i_k) \in L(n,\bc)$, then there is an injection
	$F_\bi\colon\{1,\ldots,t\} \to \{1,\ldots,n\}$ with
	$i_s = F_\bi(c_s)$ for all $s = 0,\ldots,k$.
	
	(i) Assume $t < k/2+1$.
	Using the bound $|\wn| \le \eta_n$, and the fact that $\bc$
	walks on any edge of $G(\bc)$ at least twice, we can derive
	\[ \E \biggl[\prod_{s=1}^k \bigl|w_{i_{s-1}i_s}^{(n)}\bigr|\biggr]
	\le \eta_n^{k-2(t-1)} P(S, F_\bi). \]
	Note that $\lim_{n\to\infty} \eta_n^{k-2(t-1)} = 0$ since $t < k/2+1$.
	Since the map $L(n,\bc) \to I(S,n)$ given by
	$\bi \mapsto F_\bi$ is a bijection, we have
	\begin{multline*}
		\frac{1}{n} \sum_{(i_0,\ldots,i_k) \in L(n,\bc)} \E \biggl|
		\prod_{s=1}^k w_{i_{s-1}i_s}^{(n)} \biggr|
		\le \frac{\eta_n^{k-2(t-1)}}{n} \sum_{\bi \in L(n,\bc)} P(S, F_\bi) \\
		= \frac{\eta_n^{k-2(t-1)}}{n} \sum_{F \in I(S,n)} P(S,F)
		\to 0 \qquad \text{as $n \to \infty$}
	\end{multline*}
	by Lemma \ref{lem:tree_cont}.
	
	(ii) Assume $t = k/2+1$.
	Since $S$ has $k/2$ edges and each edge of $S$ is walked on twice by $\bc$,
	we see that $S = G(\bc)$.
	As each edge of $G(\bc)$ is traversed once in each direction,
	i.e. for each $s = 1,\ldots,k$ there is an $r \in \{1,\ldots,k\}$
	with $r \ne s$ such that $c_{s-1} = c_r$ and $c_s = c_{r-1}$, we have
	\begin{multline*}
		\frac{1}{n} \sum_{(i_0,\ldots,i_k) \in L(n,\bc)} \E \biggl[
		\prod_{s=1}^k w_{i_{s-1}i_s}^{(n)} \biggr]
		= \frac{1}{n} \sum_{\bi \in L(n,\bc)} P(S,F_\bi) \\
		= \frac{1}{n} \sum_{F \in I(S,n)} P(S,F) \to 1
		\qquad \text{as $n \to \infty$}
	\end{multline*}
	by Lemma \ref{lem:tree_cont}.
\end{proof}

\begin{proof}[Proof of \eqref{eq:odd_moment_conv} and
\eqref{eq:even_moment_conv}]
	Lemma \ref{lem:no_walk_once} and \ref{lem:no_mult_cycle} tell us that
	we have \eqref{eq:iso_sum_one} if and only if $\bc$ doesn't walk on
	any edge exactly once and $t = k/2+1$.
	Otherwise, we have \eqref{eq:iso_sum_zero}.
	If $k$ is odd, then $k/2+1$ is not an integer, and so we cannot have
	$t = k/2+1$.
	So, for any odd $k \in \N$, we have
	\[ \on{ \E \tr W_n^k } \]
	by \eqref{eq:trace_sum}.
	
	Assume that $k$ is even.
	Let $U$ be the set of all $\bc \in \Gamma(k,k/2+1)$
	which traverses each edge of $G(\bc)$ twice.
	Then by Lemma \ref{lem:no_mult_cycle} (ii) and \eqref{eq:trace_sum}, we have
	\[ \lim_{n\to\infty} \frac{1}{n} \E \tr W_n^k = |U|. \]
	
	A \emph{Dyck path} of length $k$ is a finite sequence $(x_0,\ldots,x_k)$
	satisfying the following:
	\begin{enumerate}
		\item
		$x_0 = x_k = 0$;
		\item
		$x_s \ge 0$ for all $s=0,\ldots,k$;
		\item
		$|x_s - x_{s-1}| = 1$ for all $s = 1,\ldots,k$.
	\end{enumerate}
	Given an $\bc = (c_0,\ldots,c_k) \in U$, let $D(\bc) := (x_0,\ldots,x_k)$
	where $x_s$ is the distance between $1$ ($=c_0$) and $c_s$ in $G(\bc)$.
	Then it is clear that $D(\bc)$ is indeed a Dyck path, and it is
	not difficult to see that $D$ is a bijection from $U$ to the set of
	all Dyck paths of length $k$.
	It is well-known that there are exactly $\frac{1}{k/2+1}\binom{k}{k/2}$
	Dyck paths of length $k$; see \cite[Example 14.8]{vLW01}.
	Thus, we indeed have
	\[ \lim_{n\to\infty} \frac{1}{n} \E \tr W_n^k
	= \frac{1}{k/2+1} \binom{k}{k/2}. \]
	This finishes the proof of the semicircular law Theorem \ref{thm:sclaw}.
\end{proof}

\section{Gaussian convergence}
\label{sec:gaussian}
The paper \cite{Jun18} considers real symmetric random matrices
$W_1, W_2, \ldots$ with size $1 \times 1$, $2 \times 2$, \ldots
whose upper triangular entries are i.i.d.
In that paper, it is shown that if the sum of a row of $W_n$ converges
in distribution to the standard normal distribution $N(0,1)$ as $n \to \infty$,
then $\esdn \tto \scdist$ as $n \to \infty$ a.s.
We prove this fact generalized to random matrices with non-i.i.d.\ entries
in this section.
By doing so, we will demonstrate how one can apply Theorem \ref{thm:sclaw}
to obtain a semicircular law for random matrices whose entries might
have infinite variances.

The type of convergence described in the following fact will appear many times
in this section.

\begin{proposition}[Uniform convergence of triangular arrays]
	\label{prop:unif_array}
	Let $S$ be a topological space, $m_1,m_2,\ldots \in \N$,
	and $(s_{ni})_{i=1}^{m_n}$ be a finite sequence in $S$ for each $n \in \N$.
	For any $s \in S$, the following two conditions are equivalent:
	\begin{enumerate}
		\item
		$s_{ni_n} \to s$ as $n \to \infty$ for any choice of
		$i_n \in \{1,\ldots,m_n\}$ for each $n \in \N$;
		\item
		for any neighborhood $N$ of $s$, there exists some $n_0 \in \N$
		such that $s_{ni} \in N$ for all $n \ge n_0$ and $i = 1,\ldots,m_n$.
	\end{enumerate}
\end{proposition}

\begin{proof}
	We omit the straightforward proof.
\end{proof}

The following is the main result of this section.

\begin{theorem}[Gaussian convergence semicircular law]
	\label{thm:gauss_sclaw}
	For each $n \in \N$, let $W_n = (\wn)_{i,j=1}^n$ be a random
	$n \times n$ real symmetric matrix whose upper triangular entries
	are jointly independent and have symmetric distributions.
	We assume that $W_1, W_2, \ldots$ are defined on the same
	probability space.
	Assume that $(W_n)_{n \in \N}$ is a null array in the sense that
	$w_{i_nj_n}^{(n)} \tto 0$ as $n \to \infty$ for any choice of
	$i_n,j_n \in \{1,\ldots,n\}$ for each $n \in \N$.
	If also $\sum_{j=1}^n w_{i_nj}^{(n)} \tto N(0,1)$ for any choice of
	$i_n \in \{1,\ldots,n\}$ for each $n \in \N$,
	then $\esdn \tto \scdist$ as $n \to \infty$ a.s.
\end{theorem}

The following two facts will be used in the proof of
Theorem \ref{thm:gauss_sclaw}.

\begin{theorem}[Gaussian convergence]
	\label{thm:gauss_conv}
	For each $n \in \N$, let $X_{n1},\ldots,X_{nn}$ be jointly independent
	real-valued random variables.
	Assume that $X_{ni_n} \tto 0$ as $n \to \infty$ regardless of
	how we choose $i_n \in \{1,\ldots,n\}$ for each $n \in \N$.
	Then $\sum_{i=1}^n X_{ni} \tto N(0,1)$ as $n \to \infty$
	if and only if the following conditions hold:
	\begin{enumerate}
		\item
		$\lim_{n\to\infty} \sum_{i=1}^n \p(|X_{ni}|>\eps) = 0$ for all $\eps>0$;
		\item
		$\lim_{n\to\infty} \sum_{i=1}^n \E[\,X_{ni}\semicolon|X_{ni}|\le1\,]
		= 0$;
		\item
		$\lim_{n\to\infty} \sum_{i=1}^n \Var[X_{ni}\I(|X_{ni}|\le1)] = 1$.
	\end{enumerate}
\end{theorem}

\begin{proof}
	See \cite[Theorem 5.15]{Kal02}.
\end{proof}

\begin{theorem}[Bernstein's inequality]
	\label{thm:bern_ineq}
	Suppose that $X_1,\ldots,X_n$ are independent real-valued
	random variables, each with mean $0$, and each bounded by $1$.
	If $S = X_1 + \cdots + X_n$, then
	\[ \p(S \ge x)
	\le \exp \Bigl[ -\frac{x^2}{2(\E[S^2] + x)} \Bigr]
	\qquad \text{for any $x > 0$}.\]
\end{theorem}

\begin{proof}
	The proof of \cite[M20]{Bil99} with a slight change works.
\end{proof}

\begin{proof}[Proof of Theorem \ref{thm:gauss_sclaw}]
	By Theorem \ref{thm:gauss_conv}, we have
	\[ \lim_{n\to\infty} \sum_{j=1}^n \p(|w_{i_nj}^{(n)}| > \eps) = 0 \]
	for any choice of $i_n \in \{1,\ldots,n\}$ for each $n \in \N$,
	for any $\eps > 0$.
	Then Proposition \ref{prop:unif_array} implies
	\begin{equation}
		\label{eq:gauss_sclaw_prob_on}
		\on{ \sum_{i,j=1}^n \p(|\wn|>\eps) } \qquad \text{for all $\eps>0$.}
	\end{equation}
	
	Let $v_{ij}^{(n)} := \wn\I(|\wn|\le1)$ and
	$W'_n = (v_{ij}^{(n)})_{i,j=1}^n$.
	Since
	\begin{equation}
		\label{eq:gauss_sclaw_rank}
		\rank(W_n-W'_n) \le \sum_{i,j=1}^n \I(|\wn|>1)
		\le 2\sum_{1\le i\le j\le n} \I(|\wn|>1),
	\end{equation}
	by bounding $\sum_{1\le i\le j\le n} \I(|\wn|>1)$ from above
	we would be able to apply Theorem \ref{thm:rank_ineq}.
	For any given $\eps > 0$, we have some $n_0 \in \N$ such that
	\[ \sum_{1\le i\le j\le n} \p(|\wn|>1) \le \eps n/2
	\qquad \text{for all $n \ge n_0$} \]
	by \eqref{eq:gauss_sclaw_prob_on}.
	Since $\I(|\wn|>1)$, $1\le i\le j\le n$, are jointly independent,
	Bernstein's inequality (Theorem \ref{thm:bern_ineq}) implies
	\begin{equation*}
		\begin{split}
			\p\biggl( \sum_{1\le i\le j\le n} &\I(|\wn|>1) \ge \eps n \biggr) \\
			&\le \p\biggl( \sum_{1\le i\le j\le n}
			\bigl(\I(|\wn|>1)-\p(|\wn|>1)\bigr) \ge \eps n/2 \biggr) \\
			&\le \exp\biggl( -\frac{\eps^2n^2/4}
			{\sum_{1\le i\le j\le n}\p(|\wn|>1)+\eps n/2} \biggr) \\
			&\le \exp\biggl( -\frac{\eps^2n^2/4}{\eps n}\biggr)
			= \exp(-\eps n/4).
		\end{split}
	\end{equation*}
	As $\sum_{n=1}^\infty \exp(-\eps n/4) < \infty$ for each $\eps > 0$,
	we have
	\[ \on{ \sum_{1\le i\le j\le n} \I(|\wn|>1) } \qquad \text{a.s.,} \]
	and therefore
	\[ \on{ \rank(W_n-W'_n) } \qquad \text{a.s.} \]
	by \eqref{eq:gauss_sclaw_rank}.
	By Theorem \ref{thm:rank_ineq}, it now suffices to show
	$\mu_{W'_n} \tto \scdist$ as $n\to\infty$ a.s.
	
	We claim that $W'_n$ satisfies all conditions of Theorem \ref{thm:sclaw}.
	Since each $\wn$ is symmetric, each entry of $W'_n$ has mean zero.
	By Theorem \ref{thm:gauss_conv}, we have
	\[ \lim_{n\to\infty} \sum_{j=1}^n
	\Var\bigl[v_{i_nj}^{(n)}\bigr] = 1 \]
	for any choice of $i_n \in \{1,\ldots,n\}$ for each $n \in \N$.
	So, by using Proposition \ref{prop:unif_array}, we can see that
	the conditions \ref{eq:sclaw_var} and \ref{eq:sclaw_rowbdd}
	with $\wn$ replaced by $v_{ij}^{(n)}$ hold.
	Finally, the condition \eqref{eq:sclaw_lind} with $\wn$ replaced by
	$v_{ij}^{(n)}$ follows from
	\[ \sum_{i,j=1}^n \E\bigl[\,|v_{ij}^{(n)}|^2\semicolon
	|v_{ij}^{(n)}|>\eps\,\bigr] \le \sum_{i,j=1}^n
	\p\bigl(|v_{ij}^{(n)}|>\eps\bigr) \]
	and \eqref{eq:gauss_sclaw_prob_on}.
\end{proof}

\appendix
\section{Probability measures on $\R$}
\label{sec:prob_meas_r}
\subsection{Weak convergence}
\begin{definition}[The space $\Pr(\R)$]
	Let $\Pr(\R)$ denote the set of all Borel probability measures on $\R$.
	We equip $\Pr(\R)$ with the smallest topology that makes
	$\mu \mapsto \int_\R f\,d\mu$ continuous for all continuous bounded
	$f:\R\to\R$.
	Then we equip $\Pr(\R)$ with the Borel $\sigma$-algebra.
\end{definition}

Note that if $\mu,\mu_1,\mu_2,\ldots \in \Pr(\R)$,
then we have $\mu_n \tto \mu$ if and only if $\mu_n \to \mu$ under
the topology of $\Pr(\R)$.

\begin{definition}[L\'evy metric]
	If $F$ and $G$ are distribution functions, then the
	L\'evy distance between $F$ and $G$ is defined by
	\[L(F,G) := \inf\{\, \eps>0 \mid
	F(x-\eps)-\eps \le G(x) \le F(x+\eps)+\eps \text{ for all $x \in \R$}\,\}.\]
\end{definition}

It is not difficult to show that $L$ is indeed a metric on $\Pr(\R)$.
For any given $\mu \in \Pr(\R)$, let $F_\mu$ denote
the distribution function of $\mu$.

\begin{theorem}[Characterizations of weak convergence]
	\label{thm:ch_weak_conv}
	If $\mu$, $\mu_1$, $\mu_2$, $\ldots \in \Pr(\R)$, then
	the following are equivalent:
	\begin{enumerate}
		\item
		$\mu_n \tto \mu$;
		\item
		$\int_\R f_{p,q}\,d\mu_n \to \int_\R f_{p,q}\,d\mu$
		for all $p,q \in \Q$ with $p < q$, where $f_{p,q}\colon\R\to\R$
		is the function which has value $1$ on $(-\infty,p\,]$,
		has value $0$ on $[\,q,\infty)$, and is linear on $[\,p,q\,]$;
		\item
		$L(F_{\mu_n},F_\mu) \to 0$.
	\end{enumerate}
\end{theorem}

\begin{proof}
	(i) implies (ii): Directly follows from the definition of
	convergence in distribution.
	
	(ii) implies (i):
	Assume that $\int_\R f_{p,q}\,d\mu_n \to \int_\R f_{p,q}\,d\mu$
	for all $p,q \in \Q$ with $p < q$, and let $F,F_1,F_2,\ldots$ be the
	distribution functions of $\mu,\mu_1,\mu_2,\ldots$.
	Let $x$ be any continuity point of $\mu$, and let $\eps > 0$ be given.
	Since $F$ is right continuous, we have $F(x+\delta) \le F(x) + \eps$
	for some $\delta > 0$.
	If we choose any $p,q \in \Q$ with $x < p < q < x+\delta$, then
	\begin{displaymath}
		\limsup_{n \to \infty} F_n(x) \le \lim_{n \to \infty}
		\int_\R f_{p,q} \,d\mu_n = \int_\R f_{p,q} \,d\mu \le F(x+\delta)
		\le F(x) + \eps.
	\end{displaymath}
	As $F$ is also left continuous at $x$, a similar reasoning yields
	$F(x)-\eps \le \liminf_{n \to \infty} F_n(x)$.
	Since $\eps > 0$ is arbitrary, we have $F_n(x) \to F(x)$.
	
	(i) implies (iii):
	Let $\eps > 0$ be given.
	Choose continuity points $x_0,\ldots,x_k \in \R$ of $F_\mu$ such that
	$x_0 < \cdots < x_k$, $F_\mu(x_0) \le \eps$, $F_\mu(x_k) \ge 1-\eps$, and
	\begin{displaymath}
		\max\{|x_1-x_0|,\ldots,|x_k-x_{k-1}|\} \le \eps.
	\end{displaymath}
	Let $N \in \N$ be such that $n \ge N$ implies
	\begin{displaymath}
		\max\{|F_{\mu_n}(x_0) - F_\mu(x_0)|,\ldots,
		|F_{\mu_n}(x_k)-F_\mu(x_k)|\} \le \eps.
	\end{displaymath}
	Let $x \in \R$ be arbitrarily given.
	If $x > x_k$, then
	\begin{displaymath}
		F_{\mu}(x+\eps) + \eps \ge 1 \ge F_{\mu_n}(x)
	\end{displaymath}
	for any $n \in \N$.
	If $x \in (x_{i-1},x_i]$ where $i \in \{1,\ldots,n\}$, then
	\begin{displaymath}
		F_\mu(x+\eps) + \eps \ge F_\mu(x_i) + \eps \ge F_{\mu_n}(x_i)
		\ge F_{\mu_n}(x)
	\end{displaymath}
	for any $n \ge N$.
	If $x \le x_0$, then
	\begin{displaymath}
		F_\mu(x+2\eps)+2\eps \ge F_\mu(x_0) + \eps \ge F_{\mu_n}(x_0)
		\ge F_{\mu_n}(x)
	\end{displaymath}
	for any $n \ge N$.
	Similarly we can show that $F_\mu(x-2\eps)-2\eps \le F_{\mu_n}(x)$
	for all $x \in \R$ and $n \ge N$.
	So, $L(F_{\mu_n},F_\mu) \le 2\eps$ for all $n \ge N$.
	As $\eps > 0$ was arbitrary, we have $L(F_{\mu_n},F_\mu) \to 0$.
	
	(iii) implies (i):
	Let $x \in \R$ be a continuity point of $F_\mu$, and
	let $\eps > 0$ be given.
	Since $F_\mu$ is continuous at $x$, there is a $\delta \in (0,\eps/2)$
	such that $|F_\mu(x) - F_\mu(y)| \le \eps/2$ for all $|y - x| \le \delta$.
	Let $N \in \N$ be such that $L(F_{\mu_n},F_\mu) < \delta$ for all $n \ge N$.
	Then,
	\begin{displaymath}
		F_{\mu}(x-\delta) - \delta \le F_{\mu_n}(x) \le
		F_{\mu}(x+\delta) + \delta.
	\end{displaymath}
	for all $n \ge N$.
	Now observe that
	\begin{displaymath}
		F_\mu(x)-\eps \le F_\mu(x-\delta)-\eps/2 \le F_\mu(x-\delta) -
		\delta \le F_{\mu_n}(x)
	\end{displaymath}
	and
	\begin{displaymath}
		F_{\mu_n}(x) \le F_\mu(x+\delta)+\delta \le F_\mu(x+\delta) +
		\eps/2 \le F_\mu(x) + \eps
	\end{displaymath}
	for all $n \ge N$.
	Since $\eps > 0$ was arbitrary, $F_{\mu_n}(x) \to F_\mu(x)$.
\end{proof}

\subsection{Expected probability measures}

\begin{lemma}
	For any Borel $A \subset \R$, the map $e_A\colon\Pr(\R) \to [0,1]$
	defined by $e_A(\mu) := \mu(A)$ is measurable.
\end{lemma}

\begin{proof}
	For any $x \in \R$ and $\eps > 0$, let $f_{x,x+\eps}\colon\R\to\R$
	be the map which is $1$ on $(-\infty,x\,]$, $0$ on $[\,x+\eps,\infty)$,
	and linear on $[x,x+\eps]$.
	Then the map
	$\mu \mapsto \int_\R f_{x,x+\eps}\,d\mu$ is continuous, and so measurable.
	Since $\int_\R f_{x,x+1/n} \,d\mu \to \mu((-\infty,x])$
	as $n \to \infty$ by bounded convergence, the map $e_{(-\infty,x]}$ is
	measurable for any $x \in \R$.
	Let $\mathcal{C}$ be the collection of all Borel $A \subset \R$
	such that $e_A$ is measurable.
	If $A_1,A_2,\ldots \in \mathcal{C}$ are disjoint, then
	\begin{displaymath}
		e_{\bigcup_{n=1}^\infty A_n} := \sum_{n=1}^\infty e_{A_n}
	\end{displaymath}
	is measurable, and so $\bigcup_{n=1}^\infty A_n \in \mathcal{C}$.
	If $A \in \mathcal{C}$, then $e_{\R \setminus A} = 1 - e_A$ is measurable,
	and so $\R \setminus A \in \mathcal{C}$.
	These show that $\mathcal{C}$ is a $\lambda$-system containing
	$(-\infty,x\,]$ for all $x \in \R$.
	As the rays $(-\infty,x\,]$ form a $\pi$-system that generates the Borel
	$\sigma$-algebra of $\R$, the $\pi$-$\lambda$ theorem concludes the proof.
\end{proof}

\begin{theorem}
	\label{thm:expect_prob}
	Let $\mu$ be a random element of $\Pr(\R)$.
	Then there exists a unique $\E \mu \in \Pr(\R)$ satisfying
	\begin{displaymath}
		F_{\E \mu}(x) = \E[F_\mu(x)]
	\end{displaymath}
	for all $x \in \R$.
	The probability measure $\E \mu$ satisfies
	\begin{equation}
		\label{eq:expect_prob_int}
		\int_\R f \, d\E\mu = \E\left[\int_\R f \,d\mu\right]
	\end{equation}
	for all continuous and bounded $f: \R \to \R$.
\end{theorem}

\begin{proof}
	As uniqueness is easy, we only need to show the existence.
	Define $F:\R \to \R$ by $f(x) := \E[F_\mu(x)]$.
	Since $F_\mu$ is surely nondecreasing, $f$ is nondecreasing.
	Since $F_\mu(n) \to 1$ and $F_\mu(-n) \to 0$ as $n \to \infty$ surely,
	$f(n) \to 1$ and $f(-n) \to 0$ as $n \to \infty$ by bounded convergence.
	If $x_1 \ge x_2 \ge \cdots$ and $x_n \to x \in \R$, then
	$F_\mu(x_n) \to F_\mu(x)$ surely by the right continuity of distribution
	functions, and so $f(x_n) \to f(x)$ by bounded convergence.
	This shows that $f$ is right continuous, and so the proof that
	$f$ is a distribution function is finished.
	
	Let $\E \mu$ denote the Borel probability measure on $\R$
	with distribution $f$.
	For any $-\infty < a \le b < \infty$, we have
	\begin{displaymath}
		(\E \mu) ((a,b]) = f(b) - f(a) = \E[F_\mu(b) - F_\mu(a)]
		= \E[\mu((a,b])].
	\end{displaymath}
	If $(a_1,b_1], (a_2,b_2],\ldots$ are disjoint, then
	\begin{displaymath}
		(\E \mu)\biggl(\bigcup_{n=1}^\infty (a_n,b_n]\biggr)
		= \E\biggl[\mu\biggl(\bigcup_{n=1}^\infty (a_n,b_n]\biggr)\biggr]
	\end{displaymath}
	by monotone convergence.
	Since any open subset of $\R$ is a countable union of disjoint open
	intervals, and any open interval is a countable union of disjoint bounded
	intervals of the form $(a,b]$, we see that
	\begin{displaymath}
		(\E\mu)(U) = \E[\mu(U)]
	\end{displaymath}
	holds for any open $U \subset \R$.
	
	Now we show \eqref{eq:expect_prob_int}.
	By linearity of integral and expectation, we may assume that
	$f$ is nonnegative.
	For each $t \ge 0$, let $U_t := \{\,x \in \R \mid f(x) > t\,\}$.
	We want to apply Tonelli's theorem to the map $G:\Pr(\R) \times [0,\infty)
	\to [0,1]$ given by $(\nu, t) \mapsto \nu(U_t)$, so we first show that
	this map is jointly measurable.
	For each $n \in \N$, let $G_n:\Pr(\R) \times [0,\infty) \to [0,1]$ by
	\begin{displaymath}
		G_n(\nu,t) := \nu\bigl(U_{\lceil t2^{n}\rceil2^{-n}}\bigr).
	\end{displaymath}
	Since $\nu \mapsto \nu(U)$ is measurable for any open $U$,
	each $G_n$ is measurable.
	As $G_n$ increases to $G$, we can conclude that $G$ is measurable.
	Now we can use Tonelli's theorem to conclude that
	\begin{equation*}
		\begin{split}
			\int_\R f \,d\E\mu &= \int_0^\infty (\E\mu)(U_t) \,dt
			= \int_0^{\infty} \E[\mu(U_t)] \,dt \\
			&= \E\biggl[\int_0^\infty \mu(U_t)\,dt\biggr]
			= \E\biggl[\int_\R f\,d\mu\biggr].
		\end{split}
	\end{equation*}
\end{proof}

\section{Spectra of Hermitian matrices}
\label{sec:spectra}
\subsection{Basic facts}
\label{subsec:spectra_basic}
Recall the following version of the spectral theorem from linear algebra.
\begin{theorem}[Spectral Theorem]
	\label{thm:fin_spec_thm}
	Let $V$ be an $n$-dimensional complex inner product space.
	For any self-adjoint linear operator $T:V \to V$, there exists
	an orthonormal basis of $V$ consisting of eigenvectors of $V$.
\end{theorem}

\begin{proof}
	See \cite[Theorem 1.3.1]{Tao12} or
	\cite[Theorem 9 in Section 9.5]{HK71}.
\end{proof}

Also recall the following.
\begin{proposition}
	\label{prop:self_adj_ev_real}
	Any eigenvalue of a self-adjoint linear operator on
	a complex inner product space is real.
\end{proposition}

\begin{proof}
	Let $\lambda$ be an eigenvalue of a self-adjoint linear operator $T$
	on a complex inner product space $V$.
	If $Tv = \lambda v$ and $v \ne 0$, then
	\begin{displaymath}
		\bar{\lambda} \norm{v}^2 = \inner{v}{\lambda v}
		= \inner{v}{Tv} = \inner{Tv}{v} = \inner{\lambda v}{v}
		= \lambda \norm{v}^2,
	\end{displaymath}
	and so $\bar{\lambda} = \lambda$.
\end{proof}

The following naturally follows from Theorem \ref{thm:fin_spec_thm}
and Proposition \ref{prop:self_adj_ev_real}.
\begin{corollary} \label{cor:unit_diag}
	For any $n \times n$ Hermitian matrix $A$, there exists
	an $n \times n$ diagonal matrix $D$ with real entries and
	an $n \times n$ unitary matrix $U$ such that $A = UDU^\ast$.
\end{corollary}

\begin{definition}[Ordered eigenvalues]
	If $A$ is an $n \times n$ Hermitian matrix,
	then we denote the eigenvalues of $A$ counted with multiplicities as
	\begin{displaymath}
		\lambda_1(A) \ge \cdots \ge \lambda_n(A).
	\end{displaymath}
\end{definition}

\begin{definition}[Spectral distributions]
	If $A$ is an $n \times n$ Hermitian matrix, then the
	\emph{spectral distribution} of $A$ is the Borel probability measure
	on $\R$ defined by
	\[ \mu_A := \frac{1}{n} \sum_{i=1}^n \delta_{\lambda_i(A)}. \]
	We write $F_A$ as a shorthand for $F_{\mu_A}$.
\end{definition}

\begin{theorem}[Courant-Fischer minimax theorem]
	\label{thm:ev_minimax}
	Let $A$ be an $n \times n$ Hermitian matrix.
	For each $1 \le i \le n$, we have
	\begin{displaymath}
		\lambda_i(A) = \sup_{\dim V = i}
		\inf_{v \in V : \norm{v}=1} v^\ast Av
	\end{displaymath}
	and
	\begin{displaymath}
		\lambda_i(A) = \inf_{\dim V = n-i+1}
		\sup_{v \in V : \norm{v}=1} v^\ast Av,
	\end{displaymath}
	where $V$ ranges over the subspaces of $\C^n$.
\end{theorem}

\begin{proof}
	We only need to show the first equality, since the second follows
	by applying the first to $-A$.
	By the spectral theorem, we may assume that $A$ is diagonal with
	$\lambda_i(A)$ at the $(i,i)$-entry.
	Note that if $v = \begin{pmatrix} v_1 & \cdots & v_n \end{pmatrix}^T$,
	then we have $v^\ast Av = \sum_{i=1}^n \lambda_i(A)|v_i|^2$.
	If $V$ is the subspace spanned by $e_1,\ldots,e_i$,
	then $\inf_{v \in V : \norm{v}=1} v^\ast Av = \lambda_i(A)$.
	To show the other direction, let $V$ be any $i$-dimensional
	subspace of $\C^n$.
	If $W$ is the subspace spanned by $e_i,\ldots,e_n$, then
	$V \cap W \ne \{\mathbf{0}\}$ follows from
	\begin{displaymath}
		\dim V + \dim W = \dim (V \cap W) + \dim (V + W).
	\end{displaymath}
	If we choose any $w \in V \cap W$ with $\norm{w}=1$, then
	\begin{displaymath}
		\inf_{v \in V:\norm{v}=1} v^\ast Av \le w^\ast Aw \le \lambda_i(A).
	\end{displaymath}
\end{proof}

\begin{theorem}[Cauchy interlacing law]
	\label{thm:cauchy_inter}
	If $A_n$ is an $n \times n$ Hermitian matrix and $A_{n-1}$
	is the top left $(n-1) \times (n-1)$ minor of $A_n$, then
	\begin{displaymath}
		\lambda_{i+1}(A_n) \le \lambda_i(A_{n-1}) \le \lambda_i(A_n)
	\end{displaymath}
	for any $1 \le i \le n-1$.
\end{theorem}

\begin{proof}
	For any $v \in \C^{n-1}$, we have
	\begin{displaymath}
		v^\ast A_{n-1}v = \begin{pmatrix} v & 0 \end{pmatrix}
		A_{n} \begin{pmatrix} v & 0\end{pmatrix}^T.
	\end{displaymath}
	So, by Theorem \ref{thm:ev_minimax}, we have
	\begin{align*}
		\lambda_i(A_{n-1}) &= \sup_{V \subset \C^{n-1}:\dim V = i}
		\inf_{v \in V:\norm{v}=1} v^\ast A_{n-1}v \\
		&\le \sup_{V \subset \C^n:\dim V = i} \inf_{v \in V:\norm{v}=1}
		v^\ast A_nv = \lambda_i(A_n).
	\end{align*}
	By applying this result to $-A_n$, we have
	\begin{displaymath}
		\lambda_{i+1}(A_n) = -\lambda_{n-i}(-A_n)
		\le -\lambda_{n-i}(-A_{n-1}) = \lambda_i(A_{n-1}).
	\end{displaymath}
\end{proof}

\subsection{Perturbations by small Frobenius norms}
We will show that spectral distributions are stable under
two types of perturbations.
The first can be described using the following norm.

\begin{definition}
	\label{def:fnorm}
	If $A = (a_{ij})_{i,j=1}^n$ is an $n \times n$ complex matrix,
	then the \emph{Frobenius norm} of $A$ is given by
	\[\fnorm{A} := \biggl( \sum_{i,j=1}^n |a_{j}|^2 \biggr)^{1/2}.\]
\end{definition}

Note that the Frobenius norm is just the $L^2$-norm on $\C^{n^2}$.
If $A$ is a Hermitian matrix, then $\fnorm{A}^2 = \tr (A^2)$.
The following inequality tells us that the ordered tuple of eigenvalues
is stable under perturbations with small Frobenius norms.

\begin{theorem}[Hoffman-Wielandt inequality]
	\label{thm:hof-wie_ineq}
	If $A$ and $B$ are $n \times n$ Hermitian matrices, then
	\[\sum_{i=1}^n (\lambda_i(A) - \lambda_i(B))^2 \le \fnorm{A-B}^2.\]
\end{theorem}

\begin{proof}
	Recall that eigenvalues and traces are similarity invariant.
	So, by the spectral theorem, we have
	\[\sum_{i=1}^n\lambda_i(A)^2 = \tr A^2 \qquad\text{and}\qquad
	\sum_{i=1}^n\lambda_i(B)^2 = \tr B^2.\]
	Thus, it is enough to show
	\[\tr (AB) \le \sum_{i=1}^n \lambda_i(A)\lambda_i(B).\]
	(Recall that the right side of the desired inequality
	is equal to $\tr(A-B)^2$.)
	Again by the spectral theorem, we may assume that $A$ is diagonal
	with $\lambda_i(A)$ at its $i$-th entry, and write $B = UDU^\ast$
	for some unitary $U$ where $D$ is the diagonal matrix with
	$\lambda_i(B)$ as its $i$-th entry.
	If $u_{ij}$ denotes the $(i,j)$-entry of $U$, we have
	\[\tr(AB) = \tr(AUDU^\ast)
	= \sum_{i,j=1}^n |u_{ij}|^2\lambda_i(A)\lambda_j(B).\]
	It is enough to show that if $a_1 \ge \cdots \ge a_n$ and
	$b_1 \ge \cdots \ge b_n$, then the maximum of
	$\sum_{i,j=1}^n v_{ij} a_ib_j$ where $v_{ij} = v_{ji} \ge 0$ and
	$\sum_{j=1}^n v_{ij} = 1$ is obtained when $(v_{ij})_{i,j=1}^n = I$;
	that is, $v_{11} = \cdots = v_{nn} = 1$ and $v_{ij} = 0$ whenever $i \ne j$.
	Let $(v_{ij})_{i,j=1}^n$ where $v_{ij} = v_{ji} \ge 0$ and
	$\sum_{j=1}^n v_{ij} = 1$ is given.
	If $(v_{ij}) \ne I$, we have $v_{kk} < 1$ for some $k$.
	Since $\sum_{j=1}^n v_{kj} = 1$, we have $v_{k\ell} > 0$
	for some $\ell \ne k$.
	Let $w_{kk} = v_{kk} + v_{k\ell}$,
	$w_{\ell \ell} = v_{\ell \ell} + v_{k\ell} $, $w_{k\ell} = w_{\ell k} = 0$,
	and $w_{ij} = v_{ij}$ for all other $(i,j)$'s.
	Then we have $w_{ij} = w_{ji} \ge 0$ and $\sum_{j=1}^n w_{ij} = 1$.
	Also,
	\begin{equation*}
		\begin{split}
			\sum_{i,j=1}^n w_{ij}a_ib_j - \sum_{i,j=1}^n v_{ij}a_ib_j
			&= v_{k\ell}(a_kb_k + a_\ell b_\ell - a_kb_\ell - a_\ell b_k) \\
			&= v_{k\ell}(a_k-a_\ell)(b_k-b_\ell) \ge 0.
		\end{split}
	\end{equation*}
	Note that $(w_{ij})$ has more $1$'s on the diagonal than $(v_{ij})$.
	If we repeat this procedure, we will arrive at $I$,
	and this shows our claim.
\end{proof}

From the Hoffman-Wielandt inequality (Theorem \ref{thm:hof-wie_ineq}),
it follows that the spectral distribution is also stable
under perturbations of small Frobenius norms.

\begin{corollary}
	\label{cor:perturb_frob_norm}
	If $A$ and $B$ are $n \times n$ Hermitian matrices, then
	\begin{displaymath}
		[L(F_A,F_B)]^3 \le \frac{1}{n}\fnorm{A-B}^2.
	\end{displaymath}
	(For the definition of $\fnorm{\cdot}$, see Definition \ref{def:fnorm}.)
\end{corollary}

\begin{proof}
	For any $x \in \R$ and $\eps > 0$, we will show
	\begin{equation}
		\label{eq:perturb_frob_1}
		F_A(x) \le F_B(x + \eps) + \frac{1}{\eps^2n}\fnorm{A-B}^2
	\end{equation}
	and
	\begin{equation}
		\label{eq:perturb_frob_2}
		F_A(x) \ge F_B(x-\eps) - \frac{1}{\eps^2n}\fnorm{A-B}^2.
	\end{equation}
	Let $i := \bigl|\{\,\ell\mid\lambda_\ell(A) > x\,\}\bigr|$ and
	$j := \bigl|\{\,\ell\mid\lambda_\ell(B) > x+\eps\,\}\bigr|$.
	Since $\lambda_k(A) \le x$ and $\lambda_k(B) > x + \eps$
	for each $i < k \le j$, we have
	\begin{displaymath}
		(j-i)\eps^2 \le \sum_{i < k \le j} |\lambda_k(A) - \lambda_k(B)|^2
		\le \fnorm{A-B}^2.
	\end{displaymath}
	As $j-i = n(F_A(x) - F_B(x+\eps))$, \eqref{eq:perturb_frob_1} follows.
	Now let $i' := \bigl|\{\,\ell\mid\lambda_\ell(A) > x\,\}\bigr|$ and
	$j' := \bigl|\{\,\ell\mid\lambda_\ell(B) > x-\eps\,\}\bigr|$.
	Then,
	\begin{displaymath}
		(i'-j')\eps^2 \le \sum_{j' < k \le i'}|\lambda_k(A) - \lambda_k(B)|^2
		\le \fnorm{A-B}^2,
	\end{displaymath}
	and so \eqref{eq:perturb_frob_2} follows.
	
	Now let $\eps > 0$ be such that $\eps^3 := \frac{1}{n}\fnorm{A-B}^2$.
	Then, $\frac{1}{\eps^2n}\fnorm{A-B}^2 = \eps$.
	Since
	\begin{displaymath}
		F_B(x-\eps) - \eps \le F_A(x) \le F_B(x+\eps) + \eps
	\end{displaymath}
	for all $x \in \R$, we have $L(F_A,F_B) \le \eps$,
	and thus the desired claim follows.
\end{proof}

\subsection{Perturbations by small ranks}
The second type of perturbation is the low-rank perturbation.

\begin{theorem}[Rank inequality]
	\label{thm:rank_ineq}
	If $A$ and $B$ are $n \times n$ Hermitian matrices, then
	\begin{displaymath}
		\infnorm{F_{A} - F_{B}} \le \frac{\rank(A - B)}{n}
	\end{displaymath}
	where $\infnorm{f} := \sup_{x \in \R} |f(x)|$.
\end{theorem}

Note that $L(F_A,F_B) \le \infnorm{F_A - F_B}$.
\begin{proof}
	Let $k := \rank(A-B)$.
	Note that replacing $A$ and $B$ with $U AU^\ast$ and $UBU^\ast$
	for some unitary $U$ doesn't change each side of the desired inequality.
	So, using Corollary \ref{cor:unit_diag},
	we may assume that $A-B$ is diagonal.
	By swapping rows and columns, we can further assume that
	\begin{displaymath}
		A =
		\begin{pmatrix}
			A_{11} & A_{12} \\
			A_{21} & A_{22}
		\end{pmatrix}
		\text{ and }
		B =
		\begin{pmatrix}
			B_{11} & A_{12} \\
			A_{21} & A_{22}
		\end{pmatrix},
	\end{displaymath}
	where $A_{22}$ is a $(n-k) \times (n-k)$ matrix.
	If $x \in [\lambda_{i+1}(A_{22}),\lambda_i(A_{22}))$,
	then $\lambda_{k+i+1}(A) \le \lambda_{i+1}(A_{22})$ and
	$\lambda_i(A_{22}) \le \lambda_i(A)$ by the Cauchy interlacing law
	(Theorem \ref{thm:cauchy_inter}), and so
	\begin{displaymath}
		\frac{n-k-i}{n} \le F_A(x) \le \frac{n-i}{n}.
	\end{displaymath}
	By the same reasoning, we also have
	\begin{displaymath}
		\frac{n-k-i}{n} \le F_B(x) \le \frac{n-i}{n},
	\end{displaymath}
	and so
	\begin{displaymath}
		|F_A(x) - F_B(x)| \le \frac{k}{n} = \frac{\rank(A-B)}{n}.
	\end{displaymath}
	Even if $x < \lambda_{n-k}(A_{22})$ or $x \ge \lambda_1(A_{22})$,
	this inequality can be proved by a similar argument.
	Now the desired inequality follows since $x$ is arbitrary.
\end{proof}

The following is a generalization of Theorem \ref{thm:rank_ineq}.
\begin{corollary}
	\label{cor:bv_rank_ineq}
	If $A$ and $B$ are $n \times n$ Hermitian matrices and $f:\R \to \R$
	satisfies $\tvnorm{f} \le 1$, then
	\begin{displaymath}
		\left|\int_\R f \,d\mu_A - \int_\R f \,d\mu_B\right|
		\le \frac{\rank(A-B)}{n}.
	\end{displaymath}
\end{corollary}

\begin{proof}
	Let $-\infty < t_1 < \cdots < t_m < \infty$ be such that
	\begin{displaymath}
		\{t_1,\ldots,t_m\} = \{\lambda_1(A),\ldots,\lambda_n(A),
		\lambda_1(B),\ldots,\lambda_n(B)\}.
	\end{displaymath}
	Define $g:\R \to \R$ by letting $g(t_i)=f(t_i)$ for each $i=1,\ldots,m$,
	extending linearly between $t_i$ and $t_{i+1}$ for each $i=1,\ldots,m-1$,
	and setting to constants on $(-\infty,t_1]$ and $[t_m,\infty)$.
	Note that $g':\R \to \R$ exists as an integrable function, and we have
	\begin{displaymath}
		f(t_i) = f(t_m) - \int_{t_i}^\infty g'(t)\,dt
	\end{displaymath}
	for any $i=1,\ldots,m$.
	Also, as $\tvnorm{f} \le 1$, we have
	\begin{displaymath}
		\int_\R |g'(t)| \,dt = \sum_{i=1}^{n-1} |f(t_{i+1}) - f(t_i)| \le 1.
	\end{displaymath}
	Observe that
	\[\begin{split}
		\int_\R f \,d\mu_A &= \sum_{i=1}^n f(\lambda_i(A)) \\
		&= nf(t_m) - \sum_{i=1}^n \int_{\lambda_i(A)}^\infty g'(t)\,dt \\
		&= nf(t_m) - \int_\R g'(t)F_A(t) \,dt.
	\end{split}\]
	Similarly we have
	\begin{displaymath}
		\int_\R f \,d\mu_B = nf(t_m) - \int_\R g'(t)F_B(t)\,dt,
	\end{displaymath}
	and so
	\[\begin{split}
		\left| \int_\R f\,d\mu_A - \int_\R f\,d\mu_B \right|
		&\le \int_\R |g'(t)|\,dt \norm{F_A - F_B}_\infty \\
		&\le \frac{\rank(A-B)}{n}
	\end{split}\]
	by Theorem \ref{thm:rank_ineq}.
\end{proof}

\subsection{Random Hermitian matrices}
\label{subsec:rand_hermitian}
\begin{definition}[Random Hermitian matrices]
	\label{def:rand_hermitian}
	Let $H_n$ denote the space of all $n \times n$ Hermitian matrices.
	We equip $H_n$ with the standard Euclidean metric (and so the metric
	topology and the Borel $\sigma$-algebra) by identifying $H_n$ with
	$C^{n(n-1)/2}\times\R^n$, which is thought to represent the lower triangle
	of an $n \times n$ Hermitian matrix.
	A random element of $H_n$ is called a \emph{random $n \times n$
	Hermitian matrix}.
\end{definition}

From Hoffman-Wielandt inequality (Theorem \ref{thm:hof-wie_ineq}), it follows
that the map $\lambda: H_n \to \R^n$ given by
$\lambda(A) := (\lambda_1(A),\ldots,\lambda_n(A))$ is continuous.
This fact combined with the following lemma shows that the
spectral distribution of a random Hermitian matrix is measurable.

\begin{lemma}
	\label{lem:esd_conti}
	The map $\R^n \to \Pr(\R)$ given by
	\[ (x_1,\ldots,x_n) \mapsto \frac{1}{n}\sum_{i=1}^n \delta_{x_i} \]
	is continuous (and so is measurable).
\end{lemma}

\begin{proof}
	For any continuous bounded $f:\R \to \R$, the map
	$(x_1,\ldots,x_n) \in \R^n \mapsto \int_\R f \,d(\frac{1}{n}
	\sum_{i=1}^n \delta_{x_i}) = \frac{1}{n}\sum_{i=1}^n f(x_i)$ is continuous.
	Thus the given map is continuous by the definition of the topology
	of weak convergence.
\end{proof}

The following is a ``random version" of Corollary \ref{cor:perturb_frob_norm}.
If $X$ is a random $n \times n$ Hermitian matrix, we let
$\E F_X$ be the distribution function of $\E\mu_X$,
i.e. $(\E F_X)(x) = \E[F_X(x)]$.

\begin{corollary}
	\label{cor:pertrub_frob_norm_exp}
	If $X$ and $Y$ are random $n \times n$ Hermitian matrices, then
	\begin{displaymath}
		[L(\E F_X, \E F_Y)]^3
		\le \frac{1}{n}\E \bigl[\fnorm{X - Y}^2\bigr].
	\end{displaymath}
	(For the definition of $\fnorm{\cdot}$, see Definition \ref{def:fnorm}.)
\end{corollary}

\begin{proof}
	If $\E \bigl[\fnorm{X-Y}^2\bigr] = \infty$, there is nothing to prove;
	so we may assume $\E \bigl[\fnorm{X-Y}^2\bigr] < \infty$.
	By applying \eqref{eq:perturb_frob_1} and \eqref{eq:perturb_frob_2}
	to $X$ and $Y$, and taking the expectation, we have
	\begin{multline*}
		\E [F_Y(x-\eps)] - \frac{1}{\eps^2 n} \E \bigl[\fnorm{X-Y}^2\bigr]
		\le \E[F_X(x)] \\
		\le \E[F_Y(x+\eps)] + \frac{1}{\eps^2 n} \E\bigl[\fnorm{X-Y}^2\bigr]
	\end{multline*}
	for all $x \in \R$ and $\eps > 0$.
	As in the proof of Corollary \ref{cor:perturb_frob_norm},
	let $\eps > 0$ be such that
	$\eps^3 := \frac{1}{n}\E\bigl[\fnorm{X-Y}^2\bigr]$.
	Since
	\begin{displaymath}
		\E [F_Y(x-\eps)]-\eps \le \E [F_X(x)] \le \E [F_Y(x+\eps)]+\eps
	\end{displaymath}
	for all $x \in \R$, we have $L(\E F_X, \E F_Y) \le \eps$,
	and thus the desired inequality holds.
\end{proof}

\section{Concentration of measure}
\label{sec:concen_meas}
\begin{lemma}[Hoeffding's lemma]
	\label{lem:hoeffding_lemma}
	Let $-\infty < a < b < \infty$.
	If $X$ is a $[a,b]$-valued random variable, then
	\[ \E\bigl[e^{X-\E[X]}\bigr] \le \exp\bigl(2(b-a)^2\bigr). \]
\end{lemma}

\begin{proof}
	We may assume $\E[X]=0$.
	Note that $a \le 0 \le b$.
	For any $x \in [a,b]$,
	\[e^x = 1 + x + (e^c/2)x^2 \qquad \text{for some $c \in [a,b]$,}\]
	and so
	\[e^x \le 1+x+(e^{b-a}/2)x^2.\]
	Since $\E[X]=0$ and $\E[X^2]\le(b-a)^2$, we have
	\[\E[e^X] \le 1+\E[X]+(e^{b-a}/2)\E[X^2]
	\le 1+\bigl((b-a)^2/2\bigr) e^{b-a}.\]
	If $b-a < 1$, then
	\[\E[e^X] \le 1+(e/2)(b-a)^2 \le \exp\bigl((e/2)(b-a)^2\bigr).\]
	If $b-a \ge 1$, then
	\begin{multline*}
		\E[e^X]\le\bigl(1+(b-a)^2/2\bigr)e^{b-a} \\
		\le \exp\bigl((b-a)^2/2\bigr)e^{(b-a)^2} = \exp\bigl((3/2)(b-a)^2\bigr).
	\end{multline*}
	In any case, we have $\E[e^X]\le\exp\bigl(2(b-a)^2\bigr)$.
\end{proof}

\begin{theorem}[McDiarmid's inequality]
	\label{thm:mcdiarmid_ineq}
	Let $S_1,\ldots,S_n$ be measurable spaces, and
	$F \colon S_1\times\cdots\times S_n \to \R$ be a bounded
	measurable function.
	Assume that
	\[\bigl|F(x_1,\ldots,x_n) - F(x_1,\ldots,x_{i-1},x'_i,x_{i+1},\ldots,x_n)
	\bigr| \le c_i\]
	for any $i \in \{1,\ldots,n\}$, $x_j \in S_j$ for each $j\in\{1,\ldots,n\}$,
	and $x'_i \in S_i$, where $c_i > 0$ doesn't depend on $x_1,\ldots,x_n,x'_i$.
	If $X_1,\ldots,X_n$ are independent random elements of $S_1,\ldots,S_n$,
	then
	\[ \p\bigl(\bigl|F(X_1,\ldots,X_n)-\E\bigl[F(X_1,\ldots,X_n)\bigr]\bigr|
	\ge \lambda\sigma\bigr) \le 2\exp(-\lambda^2/8) \]
	for any $\lambda > 0$ where $\sigma = \sqrt{c_1^2+\cdots+c_n^2}$.
\end{theorem}

\begin{proof}
	Let us first show
	\begin{equation}
		\label{eq:mcdiarmid_first}
		\E\Bigl[\exp\bigl(tF(X_1,\ldots,X_n)\bigr)\Bigr]
		\le \exp\Bigl(2t^2\sigma^2 + t\E\bigl[F(X_1,\ldots,X_n)\bigr]\Bigr)
	\end{equation}
	for any $t > 0$ by induction on $n$.
	If $n = 0$, in which case $S_1 \times \cdots \times S_n$ is a singleton,
	$F$ is essentially a constant, and $\sigma = 0$, there is nothing to prove.
	We now proceed by induction on $n$.
	
	Note that we may assume that each $X_i$ is the projection
	$\pi_i\colon S_1 \times \cdots \times S_n \to S_i$.
	Let $\mu$ and $\mu_n$ be the distributions of
	$(X_1,\ldots,X_{n-1})$ and $X_n$.
	Let $G:S_1 \times \cdots \times S_{n-1} \to \R$ be defined by
	\[G(x_1,\ldots,x_{n-1}):=\int_{S_n} F(x_1,\ldots,x_{n-1},y) \,\mu_n(dy).\]
	For any $i \in \{1,\ldots,n-1\}$, $x_j \in S_j$ for each
	$j \in \{1,\ldots,n-1\}$, and $x'_i \in S_i$, we have
	\begin{equation*}
		\begin{split}
			\bigl|G(x_1,\ldots,x_{n-1})&-G(x_1,\ldots,x_{i-1},x'_i,x_{i+1},
			\ldots,x_{n-1})\bigr| \\
			&\le \int_{S_n}\bigl|F(x_1,\ldots,x_{n-1},y)\\
			&\qquad -F(x_1,\ldots,x_{i-1},x'_i,x_{i+1},\ldots,x_{n-1},y)\bigr|
			\,\mu_n(dy) \\
			&\le c_i.
		\end{split}
	\end{equation*}
	Since
	\begin{equation*}
		\begin{split}
			\E[&G(X_1,\ldots,X_{n-1})] \\
			&= \int_{S_1\times\cdots\times S_{n-1}}\int_{S_n}
			F(x_1,\ldots,x_{n-1},y)\,\mu_n(dy)\mu(d(x_1,\ldots,x_{n-1})) \\
			&= \E[F(X_1,\ldots,X_n)],
		\end{split}
	\end{equation*}
	the induction hypothesis implies
	\begin{equation}
		\label{eq:mcdiarmid_Gbdd}
		\begin{split}
			\E\Bigl[&\exp\bigl(tG(X_1,\ldots,X_{n-1})\bigr)\Bigr] \\
			&\le \exp\Bigl(2t^2(c_1^2+\cdots+c_{n-1}^2)+
			t\E\bigl[F(X_1,\ldots,X_n)\bigr]\Bigr).
		\end{split}
	\end{equation}
	Define $H_t:S_1 \times \cdots \times S_{n-1} \to \R$ by
	\begin{multline*}
		H_t(x_1,\ldots,x_{n-1}) := \\
		\int_{S_n} \exp\bigl(t\bigl(
		F(x_1,\ldots,x_{n-1},y)-G(x_1,\ldots,x_{n-1})\bigr)\bigr)\,\mu_n(dy).
	\end{multline*}
	Whenever $x_i \in S_i$ is fixed for each $i \in \{1,\ldots,n-1\}$, we have
	\begin{equation}
		\label{eq:mcdiarmid_Hbdd}
		H_t(x_1,\ldots,x_{n-1}) \le \exp(2t^2c_n^2)
	\end{equation}
	by Hoeffding's lemma (Lemma \ref{lem:hoeffding_lemma}).
	Thus, \eqref{eq:mcdiarmid_Hbdd} and \eqref{eq:mcdiarmid_Gbdd} yield
	\begin{equation*}
		\begin{split}
			\E\Bigl[&\exp\bigl(tF(X_1,\ldots,X_n)\bigr)\Bigr]\\
			&=\int_{S_1\times\cdots\times S_{n-1}}
			\exp\bigl(tG(x_1,\ldots,x_{n-1})\bigr) \\
			&\qquad H_t(x_1,\ldots,x_{n-1})
			\,\mu(d(x_1,\ldots,x_{n-1}))\\
			&\le \exp(2t^2c_n^2) \int_{S_1\times\cdots\times S_{n-1}}
			\exp\bigl(tG(x_1,\ldots,x_{n-1})\bigr)\,\mu(d(x_1,\ldots,x_{n-1}))\\
			&\le \exp\Bigl(2t^2\sigma^2+t\E\bigl[F(X_1,\ldots,X_n)\bigr]\Bigr),
		\end{split}
	\end{equation*}
	finishing the proof of \eqref{eq:mcdiarmid_first}.
	
	We can now finish the proof.
	Observe that
	\begin{equation*}
		\begin{split}
			\p\bigl(F(X_1,&\ldots,X_n)-\E\bigl[F(X_1,\ldots,X_n)\bigr]
			\ge \lambda\sigma\bigr)\\
			&\le e^{-t\lambda\sigma} \E\Bigl[\exp\bigl(tF(X_1,\ldots,X_n)-
			t\E\bigl[F(X_1,\ldots,X_n)\bigr]\bigr)\Bigr]\\
			&\le \exp(2t^2\sigma^2-t\lambda\sigma).
		\end{split}
	\end{equation*}
	By some calculus one can find that $t = \lambda/4\sigma$ minimizes
	the right side, yielding
	\[\p\bigl(F(X_1,\ldots,X_n)-\E\bigl[F(X_1,\ldots,X_n)\bigr]
	\ge \lambda\sigma\bigr) \le \exp(-\lambda^2/8).\]
	Applying this result to $-F$, we obtain
	\[\p\bigl(\bigl|F(X_1,\ldots,X_n)-\E\bigl[F(X_1,\ldots,X_n)\bigr]\bigr|
	\ge \lambda\sigma\bigr) \le 2\exp(-\lambda^2/8).\]
\end{proof}

The following inequality was found independently by Guntuboyina
and Leeb \cite{GL09}, and Bordenave, Caputo, and Chafa\"\i{} \cite{BCC11}.
\begin{theorem}[Concentration for spectral measures]
	\label{thm:concen_spec}
	Let $X$ be a random $n \times n$ Hermitian matrix whose
	rows of the lower triangle are jointly independent.
	If $f:\R \to \R$ satisfies $\tvnorm{f} \le 1$, and $t > 0$, then
	\[\p\biggl( \biggl| \int_\R f\,d\mu_X - \E\int_\R f\,d\mu_X \biggr| \ge t
	\biggr) \le 2\exp(-nt^2/32).\]
\end{theorem}

\begin{proof}
	Let $S_i := \C^i$ and $X_i := (X_{i1},\ldots,X_{ii})$
	for each $i=1,\ldots,n$.
	Given $(x_1,\ldots,x_n) \in S_1 \times \cdots \times S_n$,
	let $H(x_1,\ldots,x_n)$ be the Hermitian matrix whose $i$th row of
	the lower triangle is $x_i$ for each $i=1,\ldots,n$.
	Let $(x_1,\ldots,x_n) \in S_1 \times \cdots \times S_n$ and $x_i' \in S_i$.
	If we change a row or a column of a matrix,
	then the change in rank is at most $1$.
	Since $H(x_1,\ldots,x_{i-1},x_i',x_{i+1},\ldots,x_n)$ can be
	obtained from $H(x_1,\ldots,x_n)$ by changing a row and then
	changing a column, the rank of
	\begin{displaymath}
		H(x_1,\ldots,x_n) - H(x_1,\ldots,x_{i-1},x_i',x_{i+1},\ldots,x_n)
	\end{displaymath}
	is at most $2$.
	Thus, Corollary \ref{cor:bv_rank_ineq} tells us that
	\begin{displaymath}
		\left|\int_\R f\,d\mu_{H(x_1,\ldots,x_n)} -
		\int_\R f\,d\mu_{H(x_1,\ldots,x_{i-1},x_i',x_{i+1},\ldots,x_n)}
		\right| \le \frac{2}{n}.
	\end{displaymath}
	Let $X_i$ be the $i$th row of the lower triangle of $X$.
	Then, $X_1,\ldots,X_n$ are independent, and $X = H(X_1,\ldots,X_n)$.
	By applying Theorem \ref{thm:mcdiarmid_ineq} to
	$F:S_1 \times \cdots \times S_n \to \R$ given by
	$F(x_1,\ldots,x_n) := \int_\R f \,\mu_{H(x_1,\ldots,x_n)}$ and
	$X_1,\ldots,X_n$, we obtain
	\begin{displaymath}
		\p\left(\left|\int_\R f\,d\mu_X - \E\int_\R f\,d\mu_X\right|
		\ge \frac{2\lambda}{\sqrt{n}}\right) \le 2\exp(-\lambda^2/8)
	\end{displaymath}
	for any $\lambda > 0$.
	Our desired result follows by letting $\lambda = t\sqrt{n}/2$.
\end{proof}

\section{Reduction to unit variance case}
\label{sec:unit_reduction}
The Stieltjes transform method, which is the topic of the next section,
is able to prove a semicircular law (Theorem \ref{thm:unit_sclaw})
which assumes that every entry of $W_n$ has variance excatly $1/n$.
However, it seems not so easy to reduce Theorem \ref{thm:sclaw} itself
to the case the Stieltjes transform can handle.
This section provides an alternative semicircular law, which is somewhat
weaker than \ref{thm:sclaw}, that can still be reduced to what
the Stieltjes transform can handle.
If you're satisfied by the reduced version (Theorem \ref{thm:unit_sclaw}),
feel free to skip to Section \ref{sec:stieltjes}.
Otherwise, the following is the alternative semicirular law.
It was pointed out in the Remark following Theorem \ref{thm:sclaw}
as a special case of Theorem \ref{thm:sclaw}.

\begin{theorem}[A semicircular law]
	\label{thm:alter_sclaw}
	For each $n \in \N$, let $W_n=(\wn)_{i,j=1}^n$ be a random
	$n \times n$ Hermitian matrix whose upper triangular entries are
	jointly independent, have mean zero, and have finite variances.
	We assume that $W_1, W_2, \ldots$ are defined on the same probability space.
	If
	\begin{equation}
		\label{eqn:alter_var}
		\on{ \sum_{i,j=1}^n \biggl| \Varwn - \frac{1}{n} \biggr| }
	\end{equation}
	and
	\begin{equation*}
		\on{ \sum_{i,j=1}^n \lindn } \qquad \text{for every $\eps > 0$,}
	\end{equation*}
	then $\esdn \tto \scdist$ as $n \to \infty$ a.s.
\end{theorem}

\subsection{Extension of the underlying probability space}
Let $(\Omega,\mathcal{B},\p)$ be the probability space on which
$W_1, W_2, \ldots$ are defined.
If $(\Omega',\mathcal{B}',\p')$ is another probability space,
and $T\colon\Omega'\to\Omega$ is a measurable map such that
$\p(A) = \p'(T^{-1}(A))$ for all $A \in \mathcal{B}$,
then the random matrices $W_n \circ T$ satisfy all conditions of
Theorem \ref{thm:alter_sclaw}.
Assume that we proved $\mu_{W_n\circ T} \tto \scdist$ $\p'$-a.s.
Since
\[
\{\,\omega'\in\Omega' \mid \mu_{W_n\circ T(\omega')}\tto\scdist\,\}
= T^{-1}\bigl(\{\,\omega\in\Omega \mid \mu_{W_n(\omega)}\tto\scdist \,\}\bigr),
\]
we will have $\mu_{W_n} \tto \scdist$ $\p$-a.s. if we can show that
\begin{equation}
	\label{eqn:red_unit_meas}
	\{\,\omega\in\Omega \mid \mu_{W_n(\omega)}\tto\scdist \,\} \in \mathcal{B}.
\end{equation}
For any $p,q \in \Q$ with $p < q$, let $f_{p,q}$ be defined as in Theorem
\ref{thm:ch_weak_conv}.
Since $\int_\R f_{p,q}\,d\mu_{X}$ is a real-valued random variable for
any random Hermitian matrix $X$, the event
\[ \biggl\{ \lim_{n\to\infty} \int_\R f_{p,q}\,d\mu_{W_n}
= \int_\R f_{p,q}\,d\scdist \biggr\} \]
is measurable.
So, \eqref{eqn:red_unit_meas} follows from Theorem \ref{thm:ch_weak_conv},
and thus $\mu_{W_n} \tto \scdist$ $\p$-a.s.\ follows.
This shows that we can think that $W_n \circ T$'s and $\Omega'$ are the
given random matrices and the underlying space.
By considering $\Omega'=\Omega\times\{0,1\}^\N$, we may assume that
we have i.i.d.\ random variables $\xi_{ij}^{(n)}$'s,
where $n \in \N$ and $1\le i,j\le n$, independent from
$W_1,W_2,\ldots$, which satisfy
\[ \p(\xi_{ij}^{(n)}=1/\sqrt{n}) = \p(\xi_{ij}^{(n)}=-1/\sqrt{n}) = 1/2. \]

\subsection{Repeating what we already know}
The first three steps of Section \ref{sec:prelim_reductions} (that is,
until centralization) works for our case with a slight change.
Applying those steps, we can now assume the following,
and need to prove $\E\esdn \tto \scdist$.
\begin{enumerate}
	\item
	The upper triangular entries of $W_n$ are jointly independent and
	have mean zero.
	\item
	We have \eqref{eqn:alter_var}.
	\item
	There are $\eta_1,\eta_2,\ldots>0$ such that $|\wn|\le\eta_n$ and
	$\eta_n \to 0$ as $n \to \infty$.
\end{enumerate}

\subsection{Replacing and rescaling}
Let
\[
E^{(n)} := \bigl\{ (i,j) \bigm| 1 \le i,j \le n, i \ne j,
\Varwn \le \tfrac{1}{2n} \bigr\}
\]
and define $W' = (v_{ij}^{(n)})_{i,j=1}^n$ by
\[
v_{ij}^{(n)} := \begin{cases}
		\frac{1}{\bigl(n\Varwn\bigr)^{1/2}}\wn
		& \text{if $(i,j) \notin E^{(n)}$} \\
		\xi_{ij}^{(n)} & \text{if $(i,j) \in E^{(n)}$} \\
	\end{cases}.
\]
Note that
\[\begin{split}
	\frac{1}{n} \sum_{(i,j)\in E^{(n)}}
	\E \bigl[\bigl|\wn - v_{ij}^{(n)}\bigr|^2\bigr]
	&= \frac{1}{n} \sum_{(i,j)\in E^{(n)}} \Bigl( \Varwn + \frac{1}{n} \Bigr)\\
	&\le \frac{3}{2n^2} \bigl|E^{(n)}\bigr|\\
	&\le \frac{3}{n} \sum_{i,j=1}^n \Bigl| \Varwn - \frac{1}{n} \Bigr|
	\to 0
\end{split}\]
as $n \to \infty$.
Since $(a-b)^2 \le |a-b|(a+b) = |a^2-b^2|$ for any $a,b \ge 0$, we also have
\[\begin{split}
	\frac{1}{n} \sum_{(i,j)\notin E^{(n)}}
	\E \bigl[\bigl|&\wn - v_{ij}^{(n)}\bigr|^2\bigr]\\
	&= \frac{1}{n} \sum_{(i,j)\notin E^{(n)}}
	\biggl( 1 - \frac{1}{\bigl(n\Varwn\bigr)^{1/2}} \biggr)^2 \Varwn\\
	&\le \frac{1}{n} \sum_{i,j=1}^n
	\biggl( \bigl(\Varwn\bigr)^{1/2} - \frac{1}{\sqrt{n}} \biggr)^2\\
	&\le \frac{1}{n} \sum_{i,j=1}^n
	\biggl| \Varwn - \frac{1}{n} \biggr| \to 0
	\qquad \text{as $n\to\infty$.}
\end{split}\]
Combining previous two displays, we obtain
\[\on{ \sum_{i,j=1}^{n} \E\bigl[\bigl|\wn - v_{ij}^{(n)}\bigr|^2\bigr] },\]
and so it is enough to show $\E\mu_{W'_n} \tto \scdist$ by
Corollary \ref{cor:pertrub_frob_norm_exp} and Theorem \ref{thm:ch_weak_conv}.

If $i,j \notin E^{(n)}$, then
$\Var\bigl[v_{ij}^{(n)}\bigr] \ge \tfrac{1}{2n}$, and so
\[ |v_{ij}^{(n)}| \le \sqrt{2}|\wn| \le \sqrt{2}\eta_n. \]
As $|v_{ij}^{(n)}| \le 1/\sqrt{n}$ for any $(i,j) \in E^{(n)}$,
by letting $\eta'_n := \eta_n \vee (1/\sqrt{n})$ we have
$\eta'_n \to 0$ and $|v_{ij}^{(n)}| \le \eta'_n$.
Since the upper triangular entries of $W'_n$ are jointly independent
random variables with mean zero and variance $1/n$, we can now assume that
the upper triangular entries of $W_n$ have variance $1/n$, and there are
$\eta_1,\eta_2,\ldots > 0$ such that $\eta_n \to 0$ and $|\wn|\le\eta_n$.

\section{The Stieltjes transform method}
\label{sec:stieltjes}
\begin{theorem}[A unit-variance semicircular law]
	\label{thm:unit_sclaw}
	For each $n \in \N$, let $W_n = (\wn)_{i,j=1}^n$ be a random
	$n \times n$ Hermitian matrix whose upper triangular entries are
	jointly independent random variables with mean zero and variance $1/n$.
	We assume that $W_1, W_2, \ldots$ are defined on the same probability space.
	If there are $\eta_1, \eta_2 , \ldots > 0$ with $|\wn| \le \eta_n$ and
	$\eta_n \to 0$ as $n \to \infty$,
	then $\E\esdn \tto \scdist$ as $n \to \infty$ a.s.
\end{theorem}

For the readers who skipped Section \ref{sec:unit_reduction}:
note that the first step in Section \ref{sec:prelim_reductions} lets us
upgrade the conclusion of Theorem \ref{thm:unit_sclaw} to
$\esdn \tto \scdist$ as $n \to \infty$ a.s.

\subsection{Stieltjes transform}

Let $\C_+ := \{z \in \C \mid \Im z > 0\}$.
Weak convergence of probability measures on $\R$ can be coded in terms of
Stieltjes transforms.

\begin{definition}[Stieltjes transform]
	Let $\mu$ be a positive, finite Borel measure on $\R$.
	The \emph{Stieltjes transform} $s_\mu:\C_+ \to \C$ of $\mu$ is given by
	\begin{displaymath}
		s_\mu(z) := \int_\R \frac{1}{x-z} \,\mu(dx).
	\end{displaymath}
\end{definition}

Note that $|1/(x-z)| \le \Im z$ for all $x \in \R$.
So $x \mapsto 1/(x-z)$ is a continuous and bounded function,
and it also follows that $|s_\mu(z)| \le \Im z$.

\begin{theorem}[Stieltjes inversion formula]
	\label{thm:stie_inv}
	If $\mu$ is a positive, finite Borel measure on $\R$,
	then the following hold:
	\begin{enumerate}
		\item
		for any $b > 0$, $a \in \R \mapsto \frac{1}{\pi}\Im s_\mu(a + ib)$ 
		is a nonnegative function with $\int_\R \frac{1}{\pi}
		\Im s_\mu(a+ib)\,da = \mu(\R)$;
		\item
		$\frac{1}{\pi}\Im s_\mu(a+ib) \,da \tto \mu$ as $b \downarrow 0$.
		\item
		$b\Im s_\mu(ib) \to \mu(\R)$ as $b \to \infty$;
	\end{enumerate}
\end{theorem}

\begin{proof}
	If $\mu = 0$, then there is nothing to prove.
	By renormalization, we may, and will, assume $\mu(\R) = 1$.
	Note that
	\begin{equation}
	\label{eq:stieltjes_inversion_1}
		\frac{1}{\pi}\Im s_\mu(a+ib)
		= \int_\R \frac{1}{\pi}\frac{b}{(x-a)^2 + b^2} \,\mu(dx).
	\end{equation}
	Let $X$ be a real-valued random variable with distribution $\mu$,
	and $C$ be a standard Cauchy random variable
	(i.e. the law of $C$ has density $\frac{1}{\pi(x^2+1)}$) independent of $X$.
	Then $X+bC \to X$ as $b \downarrow 0$ a.s., and thus
	in distribution.
	Since the right side of \eqref{eq:stieltjes_inversion_1}
	is the density of the law of $X + bC$, both (i) and (ii) are proved.	
	As
	\begin{displaymath}
		bs_\mu(ib) = \int_\R \frac{b^2}{x^2+b^2}\,\mu(dx) \to \mu(\R)
		\qquad \text{as $b \to \infty$}
	\end{displaymath}
	by dominated convergence, (iii) is also proved.
\end{proof}

\begin{theorem}[Stieltjes continuity theorem]
	\label{thm:stie_conv}
	If $\mu,\mu_1,\mu_2,\ldots$ are Borel probability measures on $\R$,
	then $\mu_n \tto \mu$ if and only if
	$s_{\mu_n}(z) \to s_\mu(z)$ for all $z \in \C_+$.
\end{theorem}

\begin{proof}
	The ``only if" direction follows immediately from the definition
	of weak convergence.
	To show the ``if" direction, assume that $s_{\mu_n}(z) \to s_\mu(z)$
	for all $z \in \C_+$.
	Whenever $n_1 < n_2 < \cdots$ and $\mu_{n_k} \to \nu$ vaguely for some
	finite $\nu$, we have $s_{\mu_{n_k}}(z) \to s_\nu(z)$ for all
	$z \in \C_+$ since $x \mapsto 1/(x-z)$ vanishes at infinity.
	This implies $s_\nu(z) = s_\mu(z)$ for all $z \in \C_+$,
	and thus $\nu = \mu$ by Theorem \ref{thm:stie_inv}.
	As any subsequence of $(\mu_n)_{n \in \N}$ has a vaguely convergent
	further subsequence, it follows that any subsequence of $(\mu_n)_{n \in \N}$ 
	has a further subsequence converging vaguely to $\mu$.
	This shows $\mu_n \to \mu$ vaguely, and so $\mu_n \tto \mu$.
\end{proof}

\subsection{Predecessor comparison}
In the remainder of this section, we will show that
$s_{\E \mu_{W_n}}(z) \to s_{\scdist}(z)$ for all $z \in \C_+$.
To do so, we will first express $s_{\E \mu_{W_n}}(z)$ in terms of
the \emph{resolvent} $(W_n-zI)^{-1}$.
By the spectral theorem, for any Hermitian matrix $A$ the matrix $A - zI$
is invertible for any $z \in \C_+$.
Let $S_A(z) := (A - zI)^{-1}$ for any Hermitian $A$ and $z \in \C_+$.
Using the spectral theorem, we can also see that
\begin{displaymath}
	s_{\mu_A}(z) = \frac{1}{n} \tr S_A(z).
\end{displaymath}
Thus, we have
\begin{displaymath}
	s_{\E \mu_{W_n}}(z) = \E s_{\mu_{W_n}}(z) = \frac{1}{n} \E \tr S_{W_n}(z),
\end{displaymath}
and so it is enough to show that
\begin{displaymath}
	\frac{1}{n}\E\tr S_{W_n}(z) \to s_{\scdist}(z)
\end{displaymath}
for all $z \in \C_+$.

To understand the limiting behavior of $\frac{1}{n}\E\tr S_{W_n}(z)$,
we relate it with the $(n-1) \times (n-1)$ minors of $W_n$
using the Schur complement formula, which will be presented below.
For each $i \in \{1,\ldots,n\}$, let $W^{(i)}$ be the $(n-1) \times (n-1)$
matrix obtained by removing the $i$-th row and column from $W_n$.
Also, let $w_i$ denote the $i$-th column of $W_n$ with $w_{ii}$ removed.
(So, $w_i$ is an $(n-1)$-dimensional column vector.)
Let us denote the $(i,j)$-entry of a matrix $A$ by $A(i,j)$.
Recall that if $A$ is an invertible matrix, then
\begin{displaymath}
	A^{-1}(i,j) = \frac{1}{\det A} C_{ji}(A)
\end{displaymath}
where $C_{ji}(A)$ is the $(i,j)$-cofactor of $A$.
So we have
\begin{displaymath}
	S_{W_n}(z)(i,i) = \frac{\det (W^{(i)} - zI_{n-1})}{\det (W_n - zI)}
\end{displaymath}
where $I_{n-1}$ is the $(n-1) \times (n-1)$ identity matrix.

\begin{proposition}[Schur complement formula]
	\label{prop:schur_complement}
	Consider a matrix
	\begin{displaymath}
		\begin{pmatrix} A & B \\ C & D\end{pmatrix}
	\end{displaymath}
	with complex entries, where $A$ and $D$ are square matrices
	and $A$ is invertible.
	Then we have
	\begin{displaymath}
		\det \begin{pmatrix} A & B \\ C & D\end{pmatrix}
		= \det(A)\det(-CA^{-1}B+D).
	\end{displaymath}
\end{proposition}

\begin{proof}
	Note that
	\begin{displaymath}
		\begin{pmatrix} I & 0 \\ -CA^{-1} & I\end{pmatrix}
		\begin{pmatrix} A & B \\ C & D\end{pmatrix}
		= \begin{pmatrix} A & B \\ 0 & -CA^{-1}B +D\end{pmatrix}.
	\end{displaymath}
	Since
	\begin{displaymath}
		\det \begin{pmatrix} I & 0 \\ -CA^{-1} & I\end{pmatrix} = 1
	\end{displaymath}
	and
	\begin{displaymath}
		\det \begin{pmatrix} A & B \\ 0 & -CA^{-1}B +D\end{pmatrix} = \det(A)\det(-CA^{-1}B+D),
	\end{displaymath}
	we have
	\begin{displaymath}
		\det \begin{pmatrix} A & B \\ C & D\end{pmatrix} = \det(A)\det(-CA^{-1}B+D)
	\end{displaymath}
	by the multiplicativity of determinant.
\end{proof}

By Proposition \ref{prop:schur_complement}, we have
\begin{displaymath}
	\det (W_n - zI) = (-z-w_i^\ast S_{W^{(i)}}(z)w_i) \det(W^{(i)} - zI_{n-1}),
\end{displaymath}
and so
\begin{displaymath}
	S_{W_n}(z)(i,i) = \frac{-1}{z+w_i^\ast S_{W^{(i)}}(z)w_i}.
\end{displaymath}
Summing over $i=1,\ldots,n$ and taking the expectation, we obtain
\begin{equation} \label{eq:stieltjes_reduction_1}
	\frac{1}{n} \E \tr S_{W_n}(z) = \frac{1}{n} \sum_{i=1}^n
	\E \frac{-1}{z+w_i^\ast S_{W^{(i)}}(z)w_i}.
\end{equation}
(The fact that the expectation on the right side is well-defined follows
from Lemma \ref{lem:resolvent_im_pos} below.)
We will show that the right side of \eqref{eq:stieltjes_reduction_1}
gets close to
\begin{displaymath}
	\frac{-1}{z + \frac{1}{n}\E \tr S_{W_n}(z)}
\end{displaymath}
as $n$ grows, and obtain a recursive relation involving the limit of
$\frac{1}{n}\E \tr S_{W_n}(z)$.

\subsection{Derivation of a recurrence relation}
The following fact will be used repeatedly.
In particular, it will guarantee that many denominators we face
in the computation below are nonzero.
\begin{lemma} \label{lem:resolvent_im_pos}
	If $A$ is an $n \times n$ Hermitian matrix and $z \in \C_+$,
	then the following hold:
	\begin{enumerate}
		\item $\tr S_A(z) \in \C_+$;
		\item $\tr \bigl((A-zI)(A-\bar{z}I)\bigr)^{-1} \le n/(\Im z)^2$;
		\item $u^\ast S_A(z) u \in \C_+$ for any $u \in \C^n$.
	\end{enumerate}
\end{lemma}

\begin{proof}
	Let $A = U D U^\ast$ where $U$ is unitary and $D$ is
	real diagonal with diagonal entries $d_1,\ldots,d_n$.
	(i) Since trace is similarity invariant, we have
	\begin{displaymath}
		\tr S_A(z) = \tr U(D-zI)^{-1}U^\ast = \tr (D-zI)^{-1}
		= \sum_{i=1}^n \frac{1}{d_i-z} \in \C_+.
	\end{displaymath}
	(ii) Also,
	\begin{align*}
		\tr\bigl((A-zI)(A-\bar{z}I)\bigr)^{-1}
		&= \tr \bigl((D-zI)^{-1}(D-\bar{z}I)^{-1}\bigr) \\
		&= \sum_{i=1}^n \frac{1}{(d_i-z)(d_i-\bar{z})}
		= \sum_{i=1}^n \frac{1}{|d_i-z|^2} \le \frac{n}{(\Im z)^2}.
	\end{align*}
	(iii) Finally, if we let $U^\ast u = (u_1,\ldots,u_n)$, then
	\begin{displaymath}
		u^\ast S_A(z) u = (U^\ast u)^\ast (D-zI)^{-1} (U^\ast u)
		= \sum_{i=1}^n \frac{|u_i|^2}{d_i-z} \in \C_+.
	\end{displaymath}
\end{proof}

In the rest of this subsection, we fix $z \in \C_+$, and transform
\begin{displaymath}
	\frac{1}{n} \sum_{i=1}^n \E \frac{-1}{z + w_i^\ast S_{W^{(i)}}(z) w_i}
\end{displaymath}
step-by-step to obtain
\begin{displaymath}
	\frac{-1}{z + \frac{1}{n} \E \tr S_{W_n}(z)}
\end{displaymath}
(asymptotically) in the end.

\subsubsection{From $w_i^\ast S_{W^{(i)}}(z)w_i$ to $\frac{1}{n}\tr S_{W^{(i)}}(z)$}
Instead of numbering the rows and columns of $W^{(i)}$ using $1,\ldots,n-1$,
let us use $1,\ldots,\hat{i},\ldots,n$ as if $W^{(i)}$ still lies in $W_n$.
For $j,k \ne i$, let $b_{jk}$ denote the $(j,k)$-entry of $S_{W^{(i)}}(z)$.
Since
\begin{align*}
	\sum_{j,k \ne i}  |b_{jk}|^2
	&= \tr\bigl((S_{W^{(i)}}(z))^\ast S_{W^{(i)}}(z)\bigr) \\
	&= \tr\bigl([(W^{(i)}-zI_{n-1})^\ast]^{-1} (W^{(i)}-zI_{n-1})^{-1}\bigr) \\
	&= \tr \bigl((W^{(i)}-zI_{n-1})(W^{(i)}-\bar{z}I_{n-1})\bigr)^{-1},
\end{align*}
we have
\begin{equation}
	\label{eq:stieltjes_resol_sq_bdd}
	\sum_{j,k \ne i} \E |b_{jk}|^2 \le \frac{n}{(\Im z)^2}
\end{equation}
by Lemma \ref{lem:resolvent_im_pos} (ii).
The fact that $b_{jk}$'s are in $L^2$ will guarantee that all terms
in the computation below are well-defined and finite.
Using the fact that $W^{(i)}$ and $w_i$ are independent,
and each entry of $W_n$ is of mean zero and variance $1/n$, we have
\[\begin{split}
	&\E\Bigl[\Bigl|w_i^\ast S_{W^{(i)}}(z)w_i
	- \frac{1}{n} \tr S_{W^{(i)}}(z)\Bigr|^2\Bigr] \\
	&= \sum_{\substack{j,k \ne i \\ j \ne k}}
	\Bigl[\E\bigl[|w_{ji}|^2\bigr]
	\E\bigl[|w_{ki}|^2\bigr]\E\bigl[|b_{jk}|^2\bigr]
	+ \E \bigl[\overline{w_{ji}}^2\bigr] \E \bigl[w_{ki}^2\bigr]
	\E\bigl[b_{jk}\overline{b_{kj}}\bigr] \\
	&\qquad + \E\bigl[|w_{ji}|^2\bigr]\E\bigr[|w_{ki}|^2\bigr]
	\E\bigl[b_{jj}\overline{b_{kk}}\bigr]\Bigr]
	+ \sum_{j \ne i} \E\bigl[|w_{ji}|^4\bigr]\E\bigl[|b_{jj}|^2\bigr] \\
	&\qquad - \sum_{j,k \ne i} \Bigl[ \frac{1}{n}
	\E\bigl[|w_{ji}|^2\bigr]\E\bigl[b_{jj} \overline{b_{kk}}\bigr]
	+ \frac{1}{n}\E\bigl[|w_{ki}|^2\bigr]
	\E\bigl[b_{jj}\overline{b_{kk}}\bigr] \Bigr] \\
	&\qquad + \frac{1}{n^2}\sum_{j,k \ne i}
	\E\bigl[b_{jj}\overline{b_{kk}}\bigr] \\
	&= \frac{1}{n^2} \sum_{\substack{j,k \ne i \\ j \ne k}}
	\E\bigl[|b_{jk}|^2\bigr] + \sum_{j \ne i} \E\bigl[|b_{jj}|^2\bigr]
	\Bigl(\E\bigl[|w_{ji}|^4\bigr] - \frac{1}{n^2}\Bigr) \\
	&\qquad + \sum_{\substack{j,k \ne i \\ j \ne k}}
	\E \bigl[\overline{w_{ji}}^2\bigr] \E \bigl[w_{ki}^2\bigr]
	\E\bigl[b_{jk}\overline{b_{kj}}\bigr].
\end{split}\]
Note that the last term in the last line must be real.
Since
\begin{align*}
	\Bigl|\sum_{\substack{j,k \ne i \\ j \ne k}}
	\E &\bigl[\overline{w_{ji}}^2\bigr]
	\E \bigl[w_{ki}^2\bigr] \E\bigl[b_{jk}\overline{b_{kj}}\bigr] \Bigr|\\
	&\le \frac{1}{n^2} \sum_{j,k \ne i} \E\bigl|b_{jk}\overline{b_{kj}}\bigr|
	\le \frac{1}{n^2} \sum_{j,k \ne i}
	\bigl(\E\bigl[|b_{jk}|^2\bigr]\bigr)^{1/2}
	\bigl(\E\bigl[|b_{kj}|^2\bigr]\bigr)^{1/2} \\
	&\le \frac{1}{n^2} \Bigl( \sum_{j,k \ne i}
	\E \bigl[|b_{jk}|^2\bigr] \Bigr)^{1/2}
	\Bigl( \sum_{j,k \ne i} \E \bigl[|b_{jk}|^2\bigr] \Bigr)^{1/2} \\
	&= \frac{1}{n^2} \sum_{j,k \ne i} \E \bigl[|b_{jk}|^2\bigr]
\end{align*}
by the Cauchy-Schwarz inequality, we have
\[\begin{split}
	\E\Bigl[\Bigl|w_i^\ast S_{W^{(i)}}(z)w_i
	- \frac{1}{n}&\tr S_{W^{(i)}}(z)\Bigr|^2\Bigr] \\
	&\le \frac{2}{n^2} \sum_{j,k \ne i} \E \bigl[|b_{jk}|^2\bigr]
	+ \frac{\eta_n^2}{n} \sum_{j \ne i} \E\bigl[|b_{jj}|^2\bigr] \\
	&\le \left(\frac{2}{n}+\eta_n^2\right) \frac{1}{n}\sum_{j,k \ne i}
	\E \bigl[|b_{jk}|^2\bigr] \\
	&\le \left(\frac{2}{n}+\eta_n^2\right) \frac{1}{(\Im z)^2} \\
\end{split}\]
by \eqref{eq:stieltjes_resol_sq_bdd}.
It follows that
\begin{multline*}
	\frac{1}{n}\sum_{i=1}^n \E\Bigl[\Bigl|w_i^\ast S_{W^{(i)}}(z)w_i
	- \frac{1}{n} \tr S_{W^{(i)}}(z)\Bigr|\Bigr] \\
	\le \frac{1}{n}\sum_{i=1}^n \Bigl(\E\Bigl[\Bigl|w_i^\ast S_{W^{(i)}}(z)w_i
	- \frac{1}{n} \tr S_{W^{(i)}}(z)\Bigr|^2\Bigr]\Bigr)^{1/2} \\
	\le \Bigl( \Bigl(\frac{2}{n} + \eta_n^2\Bigr)
	\frac{1}{(\Im z)^2} \Bigr)^{1/2} \to 0 \qquad \text{as $n \to \infty$.}
\end{multline*}
(We are fixing $z \in \C_+$.)
Therefore, by \ref{lem:resolvent_im_pos} (i) and (iii), we have
\[\begin{split}
	\biggl|\frac{1}{n} \sum_{i=1}^n &\E \frac{-1}{z+w_i^\ast S_{W^{(i)}}(z)w_i}
	- \frac{1}{n} \sum_{i=1}^n \E \frac{-1}{z+\frac{1}{n}
	\tr S_{W^{(i)}}(z)} \biggr| \\
	&\le \frac{1}{n} \sum_{i=1}^{n} \E \biggl|\frac{w_i^\ast S_{W^{(i)}}w_i
	- \frac{1}{n}\tr S_{W^{(i)}}}{(z+w_i^\ast S_{W^{(i)}}(z)w_i)
	(z+\frac{1}{n}\tr S_{W^{(i)}}(z))} \biggr| \\
	&\le \frac{1}{n(\Im z)^2} \sum_{i=1}^n \E \Bigl|w_i^\ast S_{W^{(i)}}(z)w_i
	- \frac{1}{n}\tr S_{W^{(i)}}(z)\Bigr| \to 0
\end{split}\]
as $n \to \infty$.

\subsubsection{From $\frac{1}{n}\tr S_{W^{(i)}}(z)$ to $\frac{1}{n} \E \tr S_{W^{(i)}}(z)$}
Since the maps $x \in \R \mapsto \Re(1/(x-z))$ and $x \in \R \mapsto \Im(1/(x-z))$ have bounded variations, Theorem \ref{thm:concen_spec} implies that there are $c, C > 0$ (depending on $z$, but we are fixing $z \in \C_+$) such that
\begin{multline*}
	\p\left( \left|\frac{1}{n}\tr S_X(z) - \frac{1}{n} \E \tr S_X(z) \right|
	\ge t \right) \\
	= \p\left( \left| \int_\R \frac{1}{x-z} \,\mu_{X}(dx)
	- \E \int_\R \frac{1}{x-z} \,\mu_{X}(dx) \right| \ge t \right)
	\le c \exp(-Cnt^2)
\end{multline*}
for any $n \in \N$, $t > 0$, and a random $n \times n$ Hermitian matrix $X$.
So, we have
\[\begin{split}
	\E \bigg|\frac{1}{n}&\tr S_{W^{(i)}}(z)
	- \frac{1}{n} \E \tr S_{W^{(i)}}(z)\bigg| \\
	&= \int_0^\infty \p\left( \left|\frac{1}{n-1}\tr S_{W^{(i)}}(z)
	- \frac{1}{n-1} \E \tr S_{W^{(i)}}(z)\right|
	\ge \frac{n}{n-1}t \right) \,dt \\
	&\le \int_0^\infty c\exp\left(-\frac{Cn^2t^2}{(n-1)}\right) \,dt
	\le \int_0^\infty c \exp(-Cnt^2)\,dt
\end{split}\]
for any $n \in \N$ and $i = 1,\ldots,n$.
It follows that
\begin{multline*}
	\frac{1}{n}\sum_{i=1}^n \E\left|\frac{1}{n}\tr S_{W^{(i)}}(z)
	- \frac{1}{n}\E\tr S_{W^{(i)}}(z)\right| \\
	\le \int_0^\infty c \exp(-Cnt^2)\,dt \to 0 \qquad \text{as $n \to \infty$}
\end{multline*}
by dominated convergence.
Therefore,
\[\begin{split}
	\bigg|\frac{1}{n} \sum_{i=1}^n
	&\E \frac{-1}{z+\frac{1}{n}\tr S_{W^{(i)}}(z)}
	- \frac{1}{n} \sum_{i=1}^n \frac{-1}{z + \frac{1}{n}\E\tr S_{W^{(i)}}(z)}
	\bigg| \\
	&\le \frac{1}{n} \sum_{i=1}^n \E \left| \frac{\frac{1}{n}
	\tr S_{W^{(i)}}(z)- \frac{1}{n}\E\tr S_{W^{(i)}}(z)}{(z+\frac{1}{n}
	\tr S_{W^{(i)}}(z)) (z + \frac{1}{n}\E\tr S_{W^{(i)}}(z))} \right| \\
	&\le \frac{1}{n (\Im z)^2} \sum_{i=1}^n \E\left|\frac{1}{n}
	\tr S_{W^{(i)}}(z)- \frac{1}{n}\E\tr S_{W^{(i)}}(z)\right|
	\to 0
\end{split}\]
as $n \to \infty$.

\subsubsection{From $\frac{1}{n} \E \tr S_{W^{(i)}}(z)$ to
$\frac{1}{n} \E \tr S_{W_n}(z)$}
Let $\overline{W}^{(i)}$ be the $n \times n$ matrix obtained by
replacing all the entries in the $i$-th row and column of $W$ by $0$.
Since $\overline{W}^{(i)}$ has the same (multi)set of eigenvalues as $W^{(i)}$
except that it has one more zero eigenvalue, we have
\begin{displaymath}
	\left|\frac{1}{n} \tr S_{W^{(i)}}(z) - \frac{1}{n}
	\tr S_{\overline{W}^{(i)}}(z) \right| \le \frac{1}{n \Im z}.
\end{displaymath}
By the Hoffman-Wielandt inequality (Theorem \ref{thm:hof-wie_ineq}), we have
\[\begin{split}
	\bigg|\frac{1}{n} \tr S_{\overline{W}^{(i)}}(z)
	- \frac{1}{n} \tr S_{W_n}&(z) \bigg|
	\le \frac{1}{n} \sum_{j=1}^n \biggl|\frac{\lambda_j(\overline{W}^{(i)})
	-\lambda_j(W_n)}{\bigl(\lambda_j(\overline{W}^{(i)}) - z\bigr)
	\bigl(\lambda_j(W_n)-z\bigr)}\biggr| \\
	&\le \frac{1}{n(\Im z)^2} \sum_{j=1}^n
	\bigl|\lambda_j(\overline{W}^{(i)})-\lambda_j(W_n)\bigr| \\
	&\le \frac{1}{(\Im z)^2} \biggl( \frac{1}{n} \sum_{j=1}^n
	|\lambda_j(\overline{W}^{(i)})-\lambda_j(W_n)|^2 \biggr)^{1/2} \\
	&\le \frac{1}{(\Im z)^2} \biggl( \frac{2}{n} \sum_{j=1}^n |w_{ij}|^2
	\biggr)^{1/2}
	\le \frac{\sqrt{2}}{(\Im z)^2} \eta_n.
\end{split}\]
Combining the results of the previous two displays, we obtain
\begin{displaymath}
	\E \left| \frac{1}{n} \tr S_{W^{(i)}}(z) - \frac{1}{n} \tr S_{W_n}(z)
	\right| \le \frac{1}{n\Im z} + \frac{\sqrt{2}}{(\Im z)^2}\eta_n
\end{displaymath}
for all $n \in \N$ and $i = 1,\ldots,n$, and so
\begin{displaymath}
	\frac{1}{n} \sum_{i=1}^n \E \left| \frac{1}{n} \tr S_{W^{(i)}}(z)
	- \frac{1}{n} \tr S_{W_n}(z) \right| \le \frac{1}{n\Im z}
	+ \frac{\sqrt{2}}{(\Im z)^2}\eta_n \to 0
\end{displaymath}
as $n \to \infty$.
Therefore,
\[\begin{split}
	\biggl|\frac{1}{n} \sum_{i=1}^n & \frac{-1}{z+\frac{1}{n}
	\E\tr S_{W^{(i)}}(z)} - \frac{-1}{z + \frac{1}{n}\E\tr S_{W_n}(z)} \biggr|\\
	&= \biggl|\frac{1}{n} \sum_{i=1}^n \frac{-1}{z+\frac{1}{n}
	\E\tr S_{W^{(i)}}(z)} - \frac{1}{n} \sum_{i=1}^n \frac{-1}{z
	+ \frac{1}{n}\E\tr S_{W_n}(z)} \biggr| \\
	&\le \frac{1}{n} \sum_{i=1}^n \biggl| \frac{\frac{1}{n}
	\E\tr S_{W^{(i)}}(z)- \frac{1}{n}\E\tr S_{W_n}(z)}{(z+\frac{1}{n}
	\E\tr S_{W^{(i)}}(z)) (z + \frac{1}{n}\E\tr S_{W_n}(z))} \biggr| \\
	&\le \frac{1}{n (\Im z)^2} \sum_{i=1}^n \E\Bigl|\frac{1}{n}
	\tr S_{W^{(i)}}(z)- \frac{1}{n}\tr S_{W_n}(z)\Bigr|
	\to 0
\end{split}\]
as $n \to \infty$.

\subsubsection{The result}
Combining \eqref{eq:stieltjes_reduction_1} and the final results of the
previous three subsubsections, we obtain
\begin{equation} \label{eq:stieltjes_recurrence}
	\left| \frac{1}{n} \E \tr S_{W_n}(z) - \frac{-1}{z + \frac{1}{n}
	\E\tr S_{W_n}(z)} \right| \to 0 \qquad \text{as $n \to \infty$}
\end{equation}
for any $z \in \C_+$.

\subsection{Convergence of the Stieltjes transform}
Let $s_n(z) := \frac{1}{n} \E \tr S_{W_n}(z)$.
Fix $z \in \C_+$ for now, and let us write $s_n = s_n(z)$.
If there are $n_1 < n_2 < \cdots$ such that $|s_{n_k}| \to \infty$
as $k \to \infty$, then we would have
\begin{displaymath}
	\left|s_{n_k} - \frac{-1}{z+s_{n_k}} \right| \to \infty
	\qquad \text{as $k \to \infty$,}
\end{displaymath}
which contradicts \eqref{eq:stieltjes_recurrence}.
Thus $\{s_n \mid n \in \N\}$ is bounded, and therefore any subsequence of
$(s_n)_{n \in \N}$ has a convergent subsequence.
If we show that any convergent subsequence of $(s_n)_{n \in \N}$
should converge to a number independent of the subsequence we choose,
then we will have the convergence of $(s_n)_{n \in \N}$ to that number.

Assume $s_{n_k} \to s \in \C_+ \cup \R$ as $k \to \infty$.
Since the left side of \eqref{eq:stieltjes_recurrence} converges to
$\bigl|s+1/(z+s)\bigr|$ along $n_1 < n_2 < \cdots$, we have
\begin{displaymath}
	s + \frac{1}{z + s} = 0.
\end{displaymath}
Solving the quadratic equation, we obtain
\begin{displaymath}
	s = \frac{-z \pm \sqrt{z^2 - 4}}{2}.
\end{displaymath}

We need to decide which branch of $\sqrt{z^2 - 4}$ we use.
For simplicity, we will define $\sqrt{z^2-4}$ only for $z \in \C_+ \cup \R$,
and it will suffice.
Choose the branch of $\sqrt{z-2}$ and $\sqrt{z+2}$ defined on
$\C_+ \cup \R$ which are continuous and have nonnegative imaginary part
for all $z \in \C_+ \cup \R$.
Then let $\sqrt{z^2-4} := \sqrt{z-2} \sqrt{z+2}$.
This will make $\sqrt{z^2-4}$ continuous and have nonnegative imaginary part
on $\C_+ \cup \R$.

Since $s_{n_k} = \frac{1}{n_k} \E \tr S_{W_{n_k}}(z) \in \C_+$ for all
$k \in \N$, we have $\Im s \ge 0$.
On the other hand, $(-z-\sqrt{z^2-4})/2$ has a negative imaginary part.
Thus we have
\begin{displaymath}
	s = \frac{-z + \sqrt{z^2 - 4}}{2}.
\end{displaymath}
Since this is true for any subsequence $(s_{n_k})_{k \in \N}$ of
$(s_n)_{n \in \N}$ converging to $s$,
we have
\begin{equation}
	\label{eqn:stie_to_sc}
	s_{\E \mu_{W_n}}(z) = \frac{1}{n} \E \tr S_{W_n}(z) = s_n \to
	\frac{-z+\sqrt{z^2-4}}{2} \qquad \text{as $n \to \infty$.}
\end{equation}

\subsection{Computation of the limiting distribution}
\begin{lemma} \label{lem:sc_stie}
	If $\mu$ is a positive, finite Borel measure on $\R$ satisfying
	$s_\mu(z) = (-z + \sqrt{z^2 - 4})/2$ for all $z \in \C_+$,
	then $\mu(dx) = \sqrt{(4-x^2)_+}\,dx$.
\end{lemma}

\begin{proof}
	As
	\begin{displaymath}
		b \Im s_\mu(ib) = b\frac{-b+\sqrt{b^2+4}}{2}
		= 2\frac{-1+\sqrt{1+4/b^2}}{4/b^2}\to1 \qquad \text{as $b \to \infty$,}
	\end{displaymath}
	we have $\mu(\R) = 1$ by Theorem \ref{thm:stie_inv} (iii).
	Since $z \mapsto (-z+\sqrt{z^2-4})/2$ is continuous on
	$\C_+ \cup \R$, we have
	\begin{displaymath}
		\frac{1}{\pi}\Im s_\mu(a+ib) \to \Im \frac{-a+\sqrt{a^2-4}}{2\pi}
		= \frac{1}{2\pi} \sqrt{(4-a^2)_+} \qquad \text{as $b \downarrow 0$}
	\end{displaymath}
	for each fixed $a \in \R$.
	Since $\frac{1}{\pi} \Im s_\mu(a+ib)\,da$ is a probability density
	(by Theorem \ref{thm:stie_inv} (i)) converging pointwise to the
	probability density $\frac{1}{2\pi}\sqrt{(4-a^2)_+}$ as
	$b \downarrow 0$, we have $\frac{1}{\pi}\Im s_\mu(a+ib)\,da \tto
	\frac{1}{2\pi}\sqrt{(4-a^2)_+}\,da$ as $b \downarrow 0$ by Scheff\'e's
	theorem (\cite[Theorem 16.12]{Bil12}).
	Now $\mu(dx) = \frac{1}{2\pi}\sqrt{(4-x^2)_+}\,dx$ follows from
	Theorem \ref{thm:stie_inv} (ii).
\end{proof}

The proof of Lemma \ref{lem:sc_stie} shows how one can figure out what the
limiting spectral distribution should be in the first place.
Now we finish our proof of the semicircular law.
Since $\E\mu_{W_1}, \E\mu_{W_2}, \ldots$ are probability measures, there are
integers $n_1 < n_2 < \cdots$ such that $\E\mu_{W_{n_k}} \to \mu$ vaguely
as $k \to \infty$ for some positive, finite measure $\mu$.
Since $s_{\E\mu_{W_{n_k}}}(z) \to s_\mu(z)$ as $k \to \infty$
for all $z \in \C_+$, \eqref{eqn:stie_to_sc} implies
$s_\mu(z) = (-z+\sqrt{z^2-4})/2$ for all $z \in \C_+$.
So, we have $\mu = \scdist$ by Lemma \ref{lem:sc_stie}.
Now $\E \mu_{W_n} \tto \scdist$ follows from \eqref{eqn:stie_to_sc}
and Theorem \ref{thm:stie_conv}.
Interestingly, we were able to avoid an actual computation of $s_{\scdist}$,
in which we might have used something like the residue theorem or
the Cauchy integral formula.

\bibliographystyle{alpha}
\bibliography{wooyoung}

\end{document}